\documentclass[preprint,1p,number,sort&compress]{elsarticle}
\usepackage{amsmath,amsfonts,amsthm}
\usepackage{xcolor}

\usepackage{pgfplots}
\pgfplotsset{compat=newest}
\usetikzlibrary{plotmarks}
\usetikzlibrary{arrows.meta}
\usepgfplotslibrary{patchplots}
\usepackage{grffile}

\usepackage{hyperref}
\newcommand{\RR}{\mathbb{R}}

\newcommand{\rme}{\mathrm{e}}
\newtheorem{proposition}{Proposition}

\begin{document}
\myfooter[L]{}



\begin{frontmatter}
\title{A tensor-based exponential integrator for
       diffusion--reaction equations in common curvilinear coordinates}

\author[1]{Marco Caliari}
\ead{marco.caliari@univr.it}
\affiliation[1]{organization={Department of Computer Science, University of Verona},
                addressline={Strada le Grazie, 15}, 
                city={Verona},
                postcode={37134}, 
                country={Italy}}

\author[2]{Fabio Cassini\corref{cor1}}
\ead{fabio.cassini@unipd.it}
\affiliation[2]{organization={Department of Mathematics ``Tullio Levi-Civita'', University of Padua},
                addressline={Via Trieste, 63}, 
                city={Padova},
                postcode={35121}, 
                country={Italy}}

\cortext[cor1]{Corresponding author}

\begin{abstract}
In this paper, we study a tensor-based method for the numerical solution of
a class of diffusion--reaction equations defined on spatial domains that
admit common curvilinear coordinate representations. Typical examples
in 2D include disks (polar coordinates), and
in 3D balls or cylinders (spherical or cylindrical coordinates)
as well as spheres for problems involving the Laplace--Beltrami operator.
The proposed approach is based on a carefully chosen finite difference
discretization of the Laplace operators that yields
matrices with a structured representation as sums of Kronecker products.
For the time integration, we introduce a novel split variant of the
exponential Euler method that effectively handles the stiffness and avoids
the severe time step size restriction of classical explicit methods.
By exploiting the peculiar form of the obtained discretized operators
and the chosen splitting strategy,
we compute the needed action of the $\varphi_1$ matrix function
through suitable tensor-matrix
products in a $\mu$-mode framework.
We demonstrate the efficiency of the approach on a wide range of physically
relevant 2D and 3D examples of {\color{black}$N$-coupled}
diffusion--reaction systems
generating
Turing patterns with up to $10^6$ degrees of freedom.
\end{abstract}
\end{frontmatter}
\section{Introduction}
Diffusion--reaction equations represent one of the most elementary yet effective
approaches to mathematically model many interesting biological, physical
and chemical phenomena. We mention here, among the others, the
seminal books~\cite{B86,M03,S94}.
Despite their apparent simplicity, the presence of a nonlinear reaction term
does not
allow to determine, in general, an analytical solution. For this reason,
researchers
put a great effort to study appropriate numerical techniques to find
approximated solutions to such models~\cite{HV03}.

Concerning particularly the time discretization, one of the major difficulties
of developing
accurate and computationally feasible
numerical methods for diffusion--reaction equations stems from the linear part,
i.e., the Laplace operator. In fact, the presence of this diffusive term leads
us to address a \textit{stiff} problem, for which classical explicit time
integration
methods can not be employed unless a (typically impractical) time step size
restriction
is imposed~\cite{HW96}. On the other hand, the usage of stiff-resistant schemes
usually requires the evaluation of special functions,
such as the inverse (e.g., for implicit methods~\cite{ARW95})
or exponential-like functions
(e.g., for exponential integrators~\cite{HO10}).
{\color{black}Both cases typically involve using computationally expensive
  iterative methods at each time step, e.g., those based on Krylov subspace
  approximations.}
The efficient computation of such special functions is a crucial task to
obtain a
numerical integrator that can be used in practice. In this sense,
the geometry of the spatial domain under consideration plays a fundamental
role, since peculiar techniques can be adopted.
The Kronecker sum structure of the discretized operator on
the Cartesian product of intervals in two and three dimensions,
namely
\begin{equation}\label{eq:kronintro}
  I_2\otimes A_1 + A_2\otimes I_1 \quad \text{and} \quad
  I_3\otimes I_2\otimes A_1 + I_3\otimes A_2 \otimes I_1 +
  A_3\otimes I_2\otimes I_1,
\end{equation}
is {\color{black}an important example}. Here $A_\mu${\color{black}, $\mu=1,2,3$,}
is the discretization matrix
of a differential operator {\color{black}in one space dimension
(typically the second derivative)},
$I_\mu$ an identity matrix of suitable size,
and $\otimes$ denotes the Kronecker product.
In these cases, actions of matrix functions can be
computed with high efficiency by writing an equivalent matrix or tensor
formulation of the task~\cite{BS17,CC24,CC24bis,CC25,C24,CK20,DASS20}.
Clearly, the just mentioned tensor techniques are not valid for
{\color{black}arbitrary domains and
linear differential operators, since in general they do not admit
the Kronecker sum structure~\eqref{eq:kronintro}. However,
in some circumstances, it is possible to extend the
approaches to non-Cartesian domains, and relevant examples
comprise, for instance,
separable domains \cite{FS24},
polygonal domains
\cite{HS21}, and domains parametrized by
B-splines in the context of isogeometric analysis \cite{MJKL16}.

In the wake of this, the aim of this manuscript is to
extend tensor techniques to domains which admit
curvilinear coordinate representations. In particular,
we are interested in the numerical solution
to systems of $N$-coupled diffusion--reaction equations}
for which a component
$w\colon (0,t_*]\times \Omega\to\RR$ satisfies the prototypical Partial
  Differential Equation (PDE)
\begin{equation}\label{eq:pde}
  \left\{
  \begin{aligned}
  \partial_t w(t,\boldsymbol x) &=
  \Delta w(t,\boldsymbol x)
  + g(w(t,\boldsymbol x)),
  &t\in(0,t_*],\quad \boldsymbol x\in\Omega\subset
  \RR^{d},\\
  w(0,\boldsymbol x)&=w_0(\boldsymbol x),
  \end{aligned}\right.
\end{equation}
completed with appropriate boundary conditions.
The linear operator $\Delta$ denotes the Laplace operator
when the spatial domain $\Omega$ is a disk, a three-di\-mensional ball,
or a cylinder. On the other hand, it denotes
the Laplace--Beltrami operator
when $\Omega$ is a sphere.
In the numerical experiments,
we will also consider an anomalous diffusion case on a disk
in which the Laplace operator is substituted by an anomalous
(superdiffusive) counterpart.
The explicit form of the operator $\Delta$ in
curvilinear coordinates, for each of the
mentioned domains, will be recalled in the next section.
{\color{black}The nonlinear term $g$ accounts for the reaction
  among the different components.
The study of systems of diffusion--reaction equations}
is of great interest in the literature,
especially for the formation of the so-called Turing patterns
(that is,
spatially inhomogeneous states which may develop in the long time regime
when perturbing equilibria of the system). From the practical point of view,
the arising patterns are often
present in nature, see for instance~\cite{BKW18,DASS20,GLRS19,S79}.
While these phenomena are usually studied on the square or the cube,
we highlight that
the geometry of the domain itself may trigger different morphological classes
of patterns (see, for instance,~\cite{ATGBM02,FMS23,FSB24,LBFS17,LHAMMHH24,VAB99}).

{\color{black}
  In this paper, we also present how to couple the new matrix and tensor
  approaches with a novel
exponential integrator} for the time marching, namely a split version of the
exponential Euler method, which does not show a time step size restriction.
The remaining part of the manuscript is structured as follows.
{\color{black}In Section~\ref{sec:operators} we recall
  the relevant differential operators in curvilinear coordinates. In addition,
for ease of presentation,
  we describe the method of lines applied to the single prototypical
  equation~\eqref{eq:pde}, although the techniques straightforwardly
  apply to systems  of $N$-coupled diffusion--reaction equations.
We then proceed in Section~\ref{sec:space_disc}
by discussing an appropriate space discretization for each of the
arising one-dimensional operators.
We introduce the novel split version of the exponential Euler method and
present
how to efficiently compute the action of the
exponential-like $\varphi_1$ matrix
function in Section~\ref{sec:SExpE}.
We finally present in Section~\ref{sec:numer_exp}
several numerical experiments in two and three
space dimensions, with two- or four-coupled equations,
taken from relevant models available
in the literature, and draw the conclusions in Section~\ref{sec:conclusions}.}
\section{Semidiscretization in common curvilinear systems}\label{sec:operators}
Given the prototypical equation~\eqref{eq:pde} and the spatial domains $\Omega$
mentioned in the introduction, we recall here the differential operators
in the curvilinear systems and present their discretizations.
For the coordinate systems, we have:
\begin{subequations}
\begin{description}
\item[Polar coordinates.] If $\Omega$ is the disk of radius $\rho_*$
  {\color{black}in $\RR^2$},
  then $(\rho,\theta)\in (0,\rho_*]\times[0,2\pi)$
      and the Laplace operator becomes
\begin{equation}\label{eq:lapdisk}
  \Delta = \Delta_\mathrm{D}=\left(\frac{1}{\rho}\frac{\partial}{\partial \rho}
  +\frac{\partial^2}{\partial \rho^2}\right)
  +\frac{1}{\rho^2}\frac{\partial^2}{\partial \theta^2}.
\end{equation}
\item[Spherical coordinates.] If $\Omega$ is the 
  sphere of radius $\rho_*$ in $\RR^3$,
  then 
  $(\theta,\phi)\in[0,2\pi)\times(0,\pi)$
    {\color{black}(following the so-called mathematics convention)}
    and the Laplace--Beltrami operator becomes
\begin{equation}\label{eq:lapsphere}
  \Delta=\Delta_\mathrm{S} =
\frac{1}{\sin^2\phi}\frac{\partial^2}{\partial \theta^2}+
  \left(\cot\phi\frac{\partial}{\partial \phi}
+\frac{\partial^2}{\partial \phi^2}\right).
\end{equation}
If $\Omega$ is the three-dimensional ball (solid sphere) {\color{black}of
radius $\rho_*$ in $\RR^3$},
  then $(\rho,\theta,\phi)\in(0,\rho_*]\times[0,2\pi)\times(0,\pi)$
and the Laplace operator becomes
      \begin{equation}\label{eq:lapball}
        \Delta = \Delta_\mathrm{B}=
        \left(\frac{2}{\rho}\frac{\partial}{\partial \rho}+
        \frac{\partial^2}{\partial \rho^2}\right)
  +\frac{1}{\rho^2\sin^2\phi}\frac{\partial^2}{\partial \theta^2}
  +\frac{1}{\rho^2}\left(\cot\phi\frac{\partial}{\partial \phi}
+\frac{\partial^2}{\partial \phi^2}\right).
\end{equation}
    \item[Cylindrical coordinates] If $\Omega$ is the cylinder of
      radius $\rho_*$ and height $z_*$ {\color{black}in $\RR^3$},
      then $(\rho,\theta,z)\in(0,\rho_*]\times[0,2\pi)\times[0,z_*]$
          and the Laplace operator becomes
\begin{equation}\label{eq:lapcyl}
  \Delta = \Delta_\mathrm{C}=\left(\frac{1}{\rho}\frac{\partial}{\partial \rho}+
  \frac{\partial^2}{\partial \rho^2}\right)
  +\frac{1}{\rho^2}\frac{\partial^2}{\partial \theta^2}
  +\frac{\partial^2}{\partial z^2}.
\end{equation}
\end{description}
\end{subequations}

{\color{black}Following the approach of the method of lines, we}
now introduce the discretization points $\rho_i$,
  $\theta_j$,
    $\phi_k$, and
  $z_k$ for the spatial variables, and approximate the differential
  operators by {\color{black}centered} finite differences.
  For ease of notation, without loss of generality, in this and in the
  following section
  we consider $n$ discretization points in each coordinate.
  {\color{black}Moreover,
    we assume equispaced points with space step sizes
$h_\rho$, $h_\theta$, $h_\phi$, and $h_z$, respectively.}
  Then, PDE~\eqref{eq:pde} can be transformed into the following system
of Ordinary Differential Equations (ODEs) in \textit{vector} form
\begin{equation}\label{eq:ODE}
  \left\{
  \begin{aligned}
    \boldsymbol w'(t)&=M\boldsymbol w(t)
    +\boldsymbol g(\boldsymbol w(t))=\boldsymbol f(\boldsymbol w(t)),
    \quad t\in(0,t_*],\\
    \boldsymbol w(0)&=\boldsymbol w_0,
    \end{aligned}\right.
\end{equation}
where $\boldsymbol w(t)$ is the vector of the unknowns,
$M$ is the matrix which discretizes the {\color{black}diffusion operator
  $\Delta$},
and $\boldsymbol g(\boldsymbol w(t))$ is
{\color{black}the} nonlinear function {\color{black}corresponding to the reaction
  term $g(w(t,\boldsymbol x))$}.
We stress that we employ the expression \textit{vector} form since later
on we will solve equivalent \textit{matrix}- and \textit{tensor}-oriented
formulations.
The matrix $M$ turns out to be the
sum of Kronecker products of matrices which discretize one-dimensional
operators.
In fact, for polar coordinates on the disk we
have
\begin{subequations}\label{eq:domains}
{\color{black}
  \begin{equation}\label{eq:disk}
  \begin{aligned}
&\boldsymbol w(t)\in\RR^{n^2},
    && w_{i+(j-1)n}(t)= w(t,\rho_i,\theta_j)+\mathcal{O}(h_\rho^2)+
    \mathcal{O}(h_\theta^2),\\
&   M\in\RR^{n^2\times n^2},&&
    M=I\otimes A_{\rho,2}+
    A_\theta\otimes D_{\rho},
\end{aligned}
  \end{equation}
where $M$ is the discretized version of $\Delta_\mathrm{D}$,
while for spherical
coordinates on the sphere we have
\begin{equation}\label{eq:sphere}
  \begin{aligned}
&    \boldsymbol w(t)\in\RR^{n^2},
    && w_{j+(k-1)n}(t)= w(t,\theta_j,\phi_k)+
    \mathcal{O}(h_\theta^2)+\mathcal{O}(h_\phi^2),\\
    & M\in\RR^{n^2\times n^2},&&
    M=D_\phi\otimes A_\theta+
    A_\phi\otimes I.
\end{aligned}
\end{equation}
Here, $M$ is the discretized version of $\Delta_\mathrm{S}$.
Concerning the remaining three-dimensional cases,
for the ball in spherical coordinates we have
\begin{equation}\label{eq:ball}
  \begin{aligned}
&    \boldsymbol w(t)\in\RR^{n^3},
    && w_{i+(j-1)n+(k-1)n^2}(t)= w(t,\rho_i,\theta_j,\phi_k)+
    \mathcal{O}(h_\rho^2)+
    \mathcal{O}(h_\theta^2)+\mathcal{O}(h_\phi^2),\\
    & M\in\RR^{n^3\times n^3},&&   M=I\otimes I\otimes A_{\rho,3}
    +D_{\phi}\otimes A_\theta\otimes D_{\rho}
    +A_\phi\otimes I\otimes D_{\rho},
    \end{aligned}
\end{equation}
where $M$ is the discretized version of $\Delta_\mathrm{B}$,
while for the cylinder in cylindrical coordinates we have
\begin{equation}\label{eq:cylinder}
  \begin{aligned}
&    \boldsymbol w(t)\in\RR^{n^3},
    && w_{i+(j-1)n+(k-1)n^2}(t)=w(t,\rho_i,\theta_j,z_k)+
    \mathcal{O}(h_\rho^2)+
    \mathcal{O}(h_\theta^2)+\mathcal{O}(h_z^2),\\
& M\in\RR^{n^3\times n^3},&&  M=I\otimes I\otimes A_{\rho,2}
    +I\otimes A_\theta\otimes D_{\rho}
    +A_z\otimes I\otimes I.
    \end{aligned}
\end{equation}
Here, $M$ is the discretized version of $\Delta_\mathrm{C}$.}
\end{subequations}
In all the previous cases,
$I$ denotes the {\color{black}$n\times n$} identity matrix,
$D_{\rho}$ and $D_{\phi}$ denote the diagonal matrices of
entries $\rho_i^{-2}$ and $\sin^{-2}\phi_k$, respectively,
and the remaining $n\times n$
matrices denote discretizations of the
differential operators {\color{black}in one space dimension
  (see Table~\ref{tab:operators1d} for a summary).}
  \begin{table}
{\color{black}
    \begin{tabular}{llll}
      differential operator & spatial domain & boundary conditions & discretized operator\\
\hline
      $\frac{\partial^2}{\partial \theta^2}$ & $(0,2\pi)$ & periodic & $A_{\theta}$\\
      $\frac{d-1}{\rho}\frac{\partial}{\partial \rho}+\frac{\partial^2}{\partial \rho^2}$ $(d=2,3)$ & $(0,\rho_*)$ & not required at $0$, Neumann at $\rho_*$ & $A_{\rho,d}$\\
      $\cot\phi\frac{\partial}{\partial \phi}+\frac{\partial^2}{\partial \phi^2}$ & $(0,\pi)$ & not required & $A_{\phi}$\\
      $\frac{\partial^2}{\partial z^2}$ & $(0,z_*)$ & Neumann at 0,
      Dirichlet at $z_*$ & $A_{z}$
    \end{tabular}
    \caption{Summary of the differential operators in one space dimension with
      their spatial domain, boundary conditions and label for the discretized version.}
  \label{tab:operators1d}}
  \end{table}
The exact expression of these matrices will be derived in the next
section. Here, we simply
note that, together with the nonlinear term $\boldsymbol g$,
they take into account the boundary conditions of the
original PDE.

\section{Space discretization by finite differences}\label{sec:space_disc}
In this section, we derive appropriate space discretizations
of the above one-dimensional operators by centered
finite differences. {\color{black}We recall that the number of
  discretization points is $n$ in each coordinate and that the
  space step sizes are denoted by
$h_\rho$, $h_\theta$, $h_\phi$, and $h_z$.}

Concerning the coordinate $\theta$ {\color{black}and the corresponding differential
operator (see the first row in Table~\ref{tab:operators1d}), we perform a
discretization with second-order centered finite differences. Thus, we get}
\begin{equation}\label{eq:A_theta}
  \theta_j=j h_\theta,\quad j=1,\ldots,n,
  \quad h_\theta=\frac{2\pi}{n},\quad  A_\theta=
  \frac{1}{h_\theta^2}\begin{bmatrix}
    -2 & 1 &  &  1\\
    1 & -2 & \ddots & \\
     & \ddots & \ddots  & 1\\
    1 &  & 1 & -2
  \end{bmatrix}.
\end{equation}
This is a symmetric circulant tridiagonal matrix, with explicitly
known eigenpairs (see, e.g., \cite{BW15}), that we will
employ in each numerical experiment in Section~\ref{sec:numer_exp}.
For all the remaining matrices, namely $A_{\rho,d}$, $A_\phi$ and $A_z$,
we derive discretizations that 
yield a tridiagonal matrix
\begin{equation}\label{eq:tridiag}
A=  \begin{bmatrix}
    a_1 & b_1 &        &  \\
    c_1 & \ddots & \ddots    & \\
        & \ddots & \ddots &  b_{n-1}\\
            &        & c_{n-1} & a_n
  \end{bmatrix}.
\end{equation}
The specific values of $a_\ell$,
$b_\ell$, and $c_\ell$ clearly depend on the underlying operator
with its boundary conditions, but in all cases $A$
enjoys favorable properties which we will summarize in
Proposition~\ref{prop:tridiag}.

We start by considering the coordinate $\rho$ and
$A=A_{\rho,2}$
{\color{black}(see the second row in Table~\ref{tab:operators1d} with $d=2$).}
This will be employed
in the numerical experiments in Sections~\ref{sec:orderex},
\ref{sec:disk}, and \ref{sec:cylinder}.
We notice that
the singularity of the operator at $\rho=0$ introduced by
the coordinate system does not affect the smoothness
of our solutions. Therefore, we employ
standard centered finite differences on the computational grid
$\rho_i=\rho_1+(i-1)h_\rho$, $i=1,\dots,n$, with $\rho_1>0$ and 
$h_\rho=(\rho_*-\rho_1)/(n-1)$.
{\color{black}For the choice of the first point $\rho_1$ (and, consequently,
  of $h_\rho$) we proceed as already done in \cite{L00}, that is
  we choose $\rho_1=h_\rho/2$ (hence  $h_\rho=\rho_*/(n-1/2)$).
  In this way, if we
introduce the virtual point $\rho_0=\rho_1 -h_\rho$
and write the second-order approximation of the one-dimensional differential
operator at $\rho=\rho_1$, we get
\begin{equation*}
  \left(\frac{1}{\rho}\frac{\partial w}{\partial \rho}+
  \frac{\partial^2 w}{\partial \rho^2}\right)\bigg|_{\rho=\rho_1}=
  \frac{1}{\rho_1}\frac{w_2-w_0}{2h_\rho}+
  \frac{w_2-2w_1+w_0}{h_\rho^2}
  +\mathcal{O}(h_\rho^2)=\frac{-2w_1+2w_2}{h_\rho^2}
  +\mathcal{O}(h_\rho^2).
\end{equation*}
That is, the coefficient of the virtual value
$w_0$ is zero, without
requiring any condition at $\rho=0$.
Notice that this approach is also known as
half-point discretization (see, e.g., \cite{ATGBM02,BW15}).
Concerning the discretization at $\rho_n=\rho_*$, where
we assume that both the differential
equation and the homogeneous Neumann boundary condition are valid
(see~\cite[Sec.~I.5.3]{HV03}), we have
\begin{equation*}
  \left(\frac{1}{\rho}\frac{\partial w}{\partial \rho}+
  \frac{\partial^2 w}{\partial \rho^2}\right)\bigg|_{\rho=\rho_n}=
  \frac{2w_{n-1}-2w_n}{h_\rho^2}
  +\mathcal{O}(h_\rho^2).
\end{equation*}}%
To summarize, the discretization and the elements of the
matrix $A=A_{\rho,2}$ in \eqref{eq:tridiag} are given by
\begin{equation}\label{eq:disk_discr}
  \begin{aligned}
  \rho_i&=\frac{h_\rho}{2}+(i-1)h_\rho,
  \quad h_\rho=\frac{\rho_*}{n-\frac{1}{2}},\\
  a_\ell&=-\frac{2}{h_\rho^2},\quad
  b_\ell=\frac{2\ell}{(2\ell-1)h_\rho^2},\quad
  c_\ell=\frac{2\ell}{(2\ell+1)h_\rho^2},\ c_{n-1}=\frac{2}{h_\rho^2}.
  \end{aligned}
\end{equation}
    {\color{black}We proceed in a similar way
for the matrices $A_{\rho,3}$ and $A_\phi$, that is, we introduce virtual
points and require the coefficients of the
      virtual values to be zero. We start with
$A=A_{\rho,3}$
(see the second row in Table~\ref{tab:operators1d} with $d=3$),
which will be used in the numerical
experiment in Section~\ref{sec:ball}. After simple calculations,
we obtain $h_\rho=\rho_*/n$ and
$\rho_1=h_\rho$.} Therefore, the resulting
discretization is
\begin{equation}\label{eq:ball_discr}
\begin{aligned}
  \rho_i&=h_\rho+(i-1) h_\rho,\quad h_\rho=\frac{\rho_*}{n},\\
  a_\ell &= -\frac{2}{h_\rho^2},\quad
  b_\ell=\frac{\ell+1}{\ell h_\rho^2},\quad
  c_\ell = \frac{\ell}{(\ell+1)h_\rho^2},\ c_{n-1}=\frac{2}{h_\rho^2}.
\end{aligned}
\end{equation}
Concerning $A=A_\phi$
{\color{black}(see the third row in Table~\ref{tab:operators1d}),}
which will be used in the numerical
experiments in Sections~\ref{sec:sphere} and \ref{sec:ball},
we avoid the singularities of the operator
at $\phi=0$ and $\phi=\pi$ by introducing the symmetric computational grid
$\phi_k=\phi_1+(k-1)h_\phi$,
$k=1,\ldots,n$, with $\phi_1>0$ and $h_\phi=(\pi-2\phi_1)/(n-1)$,
and apply centered finite differences.
We assume $\phi_1 = \sigma h_\phi$ {\color{black}and require the
  coefficient of the virtual value
  to be zero.}
Then, the real number $\sigma$
and the space grid size $h_\phi$ are obtained by solving the
nonlinear system
\begin{equation}\label{eq:syscn}
\left\{\begin{aligned}
&h_\phi=\frac{\pi-2\sigma h_\phi}{n-1},\\
&-\frac{\cot (\sigma h_\phi)}{2h_\phi}+\frac{1}{h_\phi^2}=0.
\end{aligned}\right.
\end{equation}
The existence of a unique solution $(\sigma,h_\phi)$
is proved in Appendix~\ref{sec:Aphi}.
To summarize,
the discretization points and the elements~\eqref{eq:tridiag}
of the tridiagonal matrix $A=A_\phi$ are given by
\begin{equation}\label{eq:A_phi}
  \begin{aligned}
    \phi_k&=\sigma h_\phi+(k-1)h_\phi,\quad
    h_\phi=\frac{\pi}{n-1+2\sigma},\\
    a_\ell&=-\frac{2}{h_\phi^2},\quad
    b_\ell=\frac{1}{h_\phi^2}+\frac{\cot((\sigma+\ell-1)h_\phi)}{2h_\phi},\quad
    c_\ell=b_{n-\ell}.
  \end{aligned}
\end{equation}

For {\color{black} the coordinate $z$ and $A=A_z$
(see the last row in Table~\ref{tab:operators1d}),
which will be used in the numerical experiment in Section \ref{sec:cylinder},
we employ standard second-order finite differences and get}
\begin{equation}\label{eq:A_z}
  \begin{aligned}
    z_k&=(k-1)h_z,\quad h_z=\frac{z_*}{n},\\
    a_\ell&=-\frac{2}{h_z^2},\quad
    b_1=\frac{2}{h_z^2},\ b_\ell=\frac{1}{h_z^2},\quad
    c_\ell=\frac{1}{h_z^2}.
\end{aligned}
  \end{equation}
Notice that the eigenvalues and eigenvectors of this matrix can be explicitly
computed. In fact, for the $k$th eigenvalue
\begin{equation*}
  -\frac{2}{h_z^2}+\frac{2}{h_z^2}\cos\left(\frac{\pi(k-1/2)}{n}\right)
\end{equation*}
the corresponding eigenvector has $i$th component
equal to
\begin{equation*}
  \frac{\sin\frac{(n-i+1)\pi(k-1/2)}{n}}
  {\sin\frac{\pi(k-1/2)}{n}},
  \quad i=1,\ldots,n-1
\end{equation*}
and $n$th component equal to 1.
The proof follows the lines of that in \cite[Proposition 2.1]{KST99}.

As previously mentioned, in each case the derived discretization matrix enjoys
desirable properties. In fact, we have the following result.
\begin{proposition}\label{prop:tridiag}
  Let $A$ denote any of the matrices $A_\theta$,
  $A_{\rho,d}$, $A_\phi$, and
  $A_z$. Then, the following properties hold.
  \begin{enumerate}
  \item It has positive extra-diagonal elements.
  \item Its matrix exponential $\rme^{tA}$ is nonnegative for all $t\ge 0$.
  \item It is similar to a symmetric  matrix $S$.
  \item It can be diagonalized by a matrix $V$, that is $AV=V\Lambda$, with
    $V$ (weakly) well conditioned in the 2-norm.
  \item It has nonpositive eigenvalues.
  \end{enumerate}
\end{proposition}
\begin{proof}
The statements are obvious for $A=A_\theta$.

The positivity of the extra-diagonal elements of $A_{\rho,d}$
(formulas~\eqref{eq:disk_discr} and \eqref{eq:ball_discr})
and $A_z$ (formula~\eqref{eq:A_z}) is
  trivial. The positivity of $b_\ell$ and $c_\ell$
  for the matrix $A_\phi$ (see formula \eqref{eq:A_phi})
  can be easily obtained by exploiting the symmetries
of the cotangent function and the relation $\cot(\sigma h_\phi)=2/h_\phi$,
see system~\eqref{eq:syscn}.

Since all the extra-diagonal elements are nonnegative, $A$
is \emph{essentially nonnegative} and therefore its matrix exponential
$\rme^{tA}$ is nonnegative for all $t\ge 0$
(see, e.g., \cite[Thm.~I.7.2]{HV03}).

Given the matrix $\Xi$ with elements $\xi_1 = 1$ and
$\xi_{\ell+1}= \sqrt{(c_1\cdots c_{\ell})/(b_{1}\cdots b_{\ell})}$,
we have $A\Xi=\Xi S$, where $S$ is symmetric
(see~\cite[Prop.~2.1]{M25}).

Since $SQ=Q\Lambda$, where $Q$ is the orthogonal matrix of
the eigenvectors and
$\Lambda$ is the diagonal matrix of the (real) eigenvalues,
we have
\begin{equation*}
  A V = V\Lambda,\quad V=\Xi Q.
\end{equation*}
In particular, $\mathrm{cond}_2(V)\le
\mathrm{cond}_2(\Xi)$. By direct calculation, it is easy to see that
for $A=A_{\rho,d}$ the 2-norm of $\Xi$ is $1$, while
the 2-norm of its inverse is $\sqrt{2n-3}$ for $d=2$
and $n-1$ for $d=3$.
Therefore, $\mathrm{cond}_2(V)$ grows as $\sqrt{n}$ and $n$,
respectively. The condition number of the transformation matrix
of $A_\phi$ grows as $\sqrt{n}$, see Appendix~\ref{sec:Aphi}.
Hence, all these matrices are
``weakly well conditioned'', that is $\mathrm{cond}_2(V)$
grows like a small power of $n$ (see~\cite{BT92}).
Finally, the condition
number of the transformation matrix of $A_z$ is easily obtained
as $\sqrt{2}$ independently of $n$ and therefore it is ``well conditioned''.

Since the eigenvalues of the matrix are real, thanks to
Gershgorin's disk theorem they are nonpositive, thus
guaranteeing stability of the discretizations
(see~\cite[Sec.~I.3]{HV03}).
\end{proof}

\section{Split exponential Euler integrator}\label{sec:SExpE}
In this section, we introduce the novel exponential integrator that we will use
for the time integration of the stiff system of ODEs~\eqref{eq:ODE}.
The starting point is the classical first-order accurate
exponential Euler method {\color{black}(see~\cite[formula (1.6)]{HO10})}
\begin{equation}\label{eq:ee}
  \boldsymbol w_{n+1}=\boldsymbol w_n +
  \tau\varphi_1(\tau M)(M\boldsymbol w_n+\boldsymbol g(\boldsymbol w_n))=
  \boldsymbol w_n+
  \tau\varphi_1(\tau M)\boldsymbol f(\boldsymbol w_n),
\end{equation}
where $\boldsymbol w_n$ is an approximation to
$\boldsymbol w(t_n)$, $\tau$ is the time
step size (which we assume to be constant), and $\varphi_1$ is the
matrix function defined for a generic $X$ by
\begin{equation*}
  \varphi_1(X)=\sum_{i=0}^\infty \frac{X^i}{(i+1)!}=
  \int_0^1 \rme^{(1-\theta)X}d\theta.
\end{equation*}
For the cases under consideration in this manuscript, the matrix $M$ can be
written as $M_1+\dots +M_d$
(see formulas~\eqref{eq:domains}). This suggests considering the
\emph{split} approximation
  $\varphi_1(\tau M_1)\cdots\varphi_1(\tau M_d)$ to $\varphi_1(\tau M)$.
We will see in the following that
the computation of the action of the
$\varphi_1(\tau M_\mu)$ function is much simpler than that of
$\varphi_1(\tau M)$, thanks to the $\mu$-mode approach introduced
in~\cite{CC24bis,CCEOZ22,CCZ23bis} coupled with
suitable diagonalization techniques.
A similar splitting of the
$\varphi_1$ function has already been employed
in~\cite{CC24,CC24bis}. However, we highlight that in those settings
the relevant matrices did \emph{commute} (in fact, they were in the
form~\eqref{eq:kronintro}), while in our case they do
not. Nevertheless, by Taylor expansion we can easily see that
\begin{equation}\label{eq:splitorder}
  \varphi_1(\tau M)=\varphi_1(\tau M_1)\cdots
  \varphi_1(\tau M_d)+\mathcal{O}(\tau^2).
\end{equation}
Therefore, the chosen splitting
does not affect the convergence order
of the exponential integrator, since it introduces
a local error proportional to $\tau^3$.
Notice also that the order of the products in formula~\eqref{eq:splitorder}
does matter, and a different ordering would result in a different
approximation.
Throughout the manuscript, we call \emph{split exponential Euler}
the time-marching scheme
arising by inserting approximation~\eqref{eq:splitorder}
in method~\eqref{eq:ee}.
We now proceed by writing the equivalent matrix and tensor formulations of the
scheme, which in fact are the final forms of the method in each of the cases of
our interest.

We begin with the ODEs system
for the disk case \eqref{eq:disk} and
consider the action $M\boldsymbol w$ required by
the scheme, where
$\boldsymbol w$ denotes a generic column vector of length $n^2$.
We notice that
\begin{equation*}
  M\boldsymbol w =(I\otimes A_{\rho,2}+A_\theta\otimes D_\rho)\boldsymbol w
\end{equation*}
can be equivalently computed in matrix form as
\begin{equation*}
A_{\rho,2}\boldsymbol W+  D_\rho\boldsymbol WA_\theta,
\end{equation*}
where $\boldsymbol W$ is the order-2 tensor (i.e., a matrix)
whose element $w_{ij}$ coincides with the element
$w_{i+(j-1)n}$ of the vector $\boldsymbol w$.
Indeed,
this equivalence follows
from
\begin{equation}\label{eq:tucker_2d}
(L_2\otimes L_1)\boldsymbol w = \mathrm{vec}(L_1 \boldsymbol W L_2^{\sf T}),
\end{equation}
where $L_1$ and $L_2$ are two matrices of size $n \times n$ and $\mathrm{vec}$
is the operator which stacks the column of the input into
a suitable column
(see~\cite[formula (2)]{VL00}),
and taking into account that $A_\theta$
is a symmetric matrix.
We now
discuss the computation of the action $\varphi_1(\tau M_2)\boldsymbol w$ with
\begin{equation*}
M_2=A_\theta\otimes D_\rho.
\end{equation*}
Since $A_\theta$ can be written as
$Q_\theta\Lambda_\theta Q_\theta^{\sf T}$, where $Q_\theta$ is
the orthogonal matrix of the eigenvectors and $\Lambda_\theta$ the diagonal
matrix of the eigenvalues ${\lambda_{\theta}}_j$,
by exploiting the mixed-product property of $\otimes$ we get
\begin{equation*}
  A_\theta\otimes D_\rho=
  (Q_\theta \otimes I)(\Lambda_\theta \otimes D_\rho)
  (Q_\theta^{\sf T}\otimes I).
\end{equation*}
Hence, the $\varphi_1$ matrix function can be conveniently expressed as
\begin{equation*}
  \varphi_1(\tau (A_\theta\otimes D_\rho))=
  (Q_\theta \otimes I)\varphi_1(\tau (\Lambda_\theta \otimes D_\rho))
  (Q_\theta^{\sf T}\otimes I).
\end{equation*}
Concerning the application to a vector $\boldsymbol w$, from
equivalence~\eqref{eq:tucker_2d} we have
the matrix form $\boldsymbol WQ_\theta$ of
$(Q^{\sf T}_\theta\otimes I)\boldsymbol w$.
Then, the diagonal matrix function
$\varphi_1(\tau(\Lambda_\theta\otimes D_\rho))$
is computed by exploiting
the \textit{scalar} formula
\begin{equation*}
  \varphi_1(x)=\left\{\begin{aligned}
  &  \frac{\rme^x-1}{x}, && x\ne 0,\\
  & \quad \ 1, && x=0.
  \end{aligned}\right.
\end{equation*}
The result is stored in the order-2 tensor
$\boldsymbol \Phi_{\rho,\theta}$ with element
$\varphi_1(\tau {\rho_i^{-2}\lambda_\theta}_j)$ in position $(i,j)$.
Hence, the computation $\varphi_1(\tau(\lambda_\theta\otimes D_\rho))
(Q^{\sf T}_\theta\otimes I)\boldsymbol w$ is realized by
$\boldsymbol \Phi_{\rho,\theta}\circ (\boldsymbol WQ_\theta)$,
where $\circ$ denotes the Hadamard product. By using again
equivalence~\eqref{eq:tucker_2d}, we get
\begin{equation*}
  \varphi_1(\tau M_2)\boldsymbol w=
  \mathrm{vec}\left(\big(\boldsymbol \Phi_{\rho,\theta}\circ (\boldsymbol WQ_\theta)\big)Q_\theta^{\sf T}\right).
\end{equation*}
The final requirement is the computation
of the action of $\varphi_1(\tau M_1)\boldsymbol w$ in matrix form, where
\begin{equation*}
  M_1=I\otimes A_{\rho,2}.
\end{equation*}
This is realized by using the identity
\begin{equation*}
  \varphi_1(\tau (I\otimes A_{\rho,2}))=
  I\otimes\varphi_1(\tau A_{\rho,2}).
\end{equation*}
and therefore, again by formula~\eqref{eq:tucker_2d}, we have
\begin{equation*}
  \varphi_1(\tau M_1)\boldsymbol w=\mathrm{vec}\left(
  \varphi_1(\tau A_{\rho,2})\boldsymbol W\right).
\end{equation*}
Notice that the computation of the matrix $\varphi_1(\tau A_{\rho,2})$
can be conveniently performed by diagonalization,
since the transformation matrix of $A_{\rho,2}$ is weakly well
conditioned (see Proposition~\ref{prop:tridiag}).
In summary, the split exponential Euler scheme for the disk
case~\eqref{eq:disk} is
realized in matrix form as
\begin{subequations}\label{eq:sexpe}
\begin{equation}\label{eq:ee_disk}
  \left\{\begin{aligned}
  \boldsymbol F_{n}&=A_{\rho,2}\boldsymbol W_n+D_\rho\boldsymbol W_nA_\theta+
\boldsymbol G_n,\\
  \boldsymbol W_{n+1}&=\boldsymbol W_{n}+\tau
\varphi_1(\tau A_{\rho,2})\big(\boldsymbol \Phi_{\rho,\theta}
\circ(\boldsymbol F_nQ_\theta)\big)Q_\theta^{\sf T},
\end{aligned}\right.
\end{equation}
where $\boldsymbol W_n$ and
$\boldsymbol G_n$ are the order-2 tensors corresponding to
the vectors $\boldsymbol w_n$ and $\boldsymbol g(\boldsymbol w_n)$,
respectively.

The ODEs system for the sphere case \eqref{eq:sphere} can be treated in a
similar way.
In fact, the action of $M\boldsymbol w$ is computed as
\begin{equation*}
  A_\theta\boldsymbol WD_\phi+\boldsymbol WA_\phi^{\sf T}.
\end{equation*}
Then, for the matrix function $\varphi_1(\tau M_2)$, we have
\begin{equation*}
\varphi_1(\tau M_2)=  \varphi_1\left(\tau(A_\phi\otimes I)\right)=
  \varphi_1\left(\tau A_\phi\right)\otimes I
\end{equation*}
and therefore we compute its action on $\boldsymbol w$ as
\begin{equation*}
  \boldsymbol W \varphi_1(\tau A_\phi)^{\sf T},
  \end{equation*}
where the matrix $\varphi_1(\tau A_\phi)$
can be computed by diagonalization.
Moreover, the action $\varphi_1(\tau M_1)\boldsymbol w$ is realized by
\begin{equation*}
Q_\theta\big(\boldsymbol \Phi_{\theta,\phi}\circ(Q_\theta^{\sf T}\boldsymbol W)\big),
\end{equation*}
where 
$\boldsymbol \Phi_{\theta,\phi}$ is the order-2 tensor with element
$\varphi_1(\tau {\lambda_\theta}_j\sin^{-2} \phi_k)$
in position $(j,k)$.
The split exponential Euler integrator in the sphere case
\eqref{eq:sphere} is thus written in matrix
form as
\begin{equation}\label{eq:ee_sphere}
\left\{\begin{aligned}
\boldsymbol F_{n}&=A_{\theta}\boldsymbol W_nD_\phi+
\boldsymbol W_nA_\phi^{\sf T}+
\boldsymbol G_n,\\
    \boldsymbol W_{n+1}&=\boldsymbol W_{n}+\tau
Q_\theta\big(\boldsymbol \Phi_{\theta,\phi}\circ(Q_\theta^{\sf T}\boldsymbol F_n\varphi_1(\tau A_\phi)^{\sf T})\big).
\end{aligned}\right.
  \end{equation}

We now consider the ODEs system  for the three-dimensional ball
case \eqref{eq:ball}.
First of all, for a column vector $\boldsymbol w$ of $n^3$ elements, the
action $M\boldsymbol w$ can be equivalently computed in tensor form as
\begin{equation*}
\boldsymbol W\times_1 A_{\rho,3}+
  \boldsymbol W\times_1 D_\rho\times_2 A_\theta\times_3 D_\phi+
  \boldsymbol W\times_1 D_\rho\times_3 A_\phi,
\end{equation*}
where $\boldsymbol W$ is the order-3 tensor whose element $w_{ijk}$
coincides with the element $i+(j-1)n+(k-1)n^2$
of the vector $\boldsymbol w$.
The symbol $\times_1$ denotes the 1-mode matrix-tensor product (multiplication
of the matrix onto the columns of the tensor), $\times_2$ the 2-mode
product (multiplication
onto the rows of the tensor), and
$\times_3$ the
3-mode product (multiplication onto the tubes of the tensor),
see~\cite[Sec.~2.5]{KB09}.
The equivalence follows from
the natural generalization of formula~\eqref{eq:tucker_2d}, that is
\begin{equation*}
  (L_3\otimes L_2\otimes L_1)\boldsymbol w =
  \mathrm{vec}(\boldsymbol W\times_1 L_1\times_2 L_2\times_3 L_3).
\end{equation*}
Here, $L_\mu$ are matrices
of size $n\times n$, and the concatenation of $\mu$-mode products is
known as the \emph{Tucker operator}.
For more details on the $\mu$-mode
framework and its implementation for $d=3$ by highly efficient
level 3 BLAS, we invite the reader
to consult, e.g., \cite{CCZ23bis}.
Concerning the evaluation of the needed actions of matrix functions, we have
\begin{equation*}
  \varphi_1(\tau M_3)=
  \varphi_1(\tau (A_\phi\otimes I\otimes D_\rho))=
  (V_\phi\otimes I\otimes I)\varphi_1(\tau(\Lambda_\phi\otimes I\otimes
  D_\rho))(V_\phi^{-1}\otimes I\otimes I),
\end{equation*}
where we used $A_\phi =V_\phi\Lambda_\phi V_\phi^{-1}$.
Hence, the action $\varphi_1(\tau M_3)\boldsymbol w$ is realized as
\begin{equation*}
\big(\boldsymbol \Phi_{\rho,\cdot,\phi}
\circ (\boldsymbol W\times_3 V_\phi^{-1})\big)\times_3 V_\phi,
\end{equation*}
where $\boldsymbol \Phi_{\rho,\cdot,\phi}$ in the order-3 tensor with element
$\varphi_1(\tau\rho_i^{-2}{\lambda_\phi}_k)$ in position $(i,j,k)$,
being ${\lambda_\phi}_k$ the eigenvalues of $A_\phi$.
Then, for $M_2$ we have
\begin{equation*}
  \varphi_1(\tau M_2)=
    \varphi_1(\tau (D_\phi\otimes A_\theta\otimes D_\rho))=
    (I\otimes Q_\theta\otimes I)
    \varphi_1(\tau(D_\phi\otimes \Lambda_\theta\otimes D_\rho))
    (I\otimes Q_\theta^{\sf T}\otimes I)
\end{equation*}
and hence its action on $\boldsymbol w$ is realized by
\begin{equation*}
  \big(\boldsymbol \Phi_{\rho,\theta,\phi}
  \circ (\boldsymbol W\times_2 Q_\theta^{\sf T})\big)\times_2 Q_\theta,
\end{equation*}
where $\boldsymbol \Phi_{\rho,\theta,\phi}$ in the order-3 tensor with element
$\varphi_1(\tau \rho_i^{-2}{\lambda_\theta}_j\sin^{-2} \phi_k)$
in position $(i,j,k)$.
Finally, from the equivalence
\begin{equation*}
  \varphi_1(\tau M_1)
  =\varphi_1(\tau(I\otimes I\otimes A_{\rho,3}))=
I\otimes I\otimes \varphi_1(\tau A_{\rho,3}),
\end{equation*}
the action $\varphi_1(\tau M_1)\boldsymbol w$ is computed as
\begin{equation*}
  \boldsymbol W\times_1 \varphi_1(\tau A_{\rho,3}),
  \end{equation*}
where the matrix $\varphi_1(\tau A_{\rho,3})$
can be computed by diagonalization, since the transformation matrix
of $A_{\rho,3}$ is weakly well conditioned.
Therefore, concatenating the three matrix function actions described above,
the split exponential Euler method for the ball case~\eqref{eq:ball}
is realized in tensor form as
\begin{equation}\label{eq:ee_ball}
  \left\{\begin{aligned}
\boldsymbol F_n&=\boldsymbol W_n\times_1 A_{\rho,3}+
  \boldsymbol W_n\times_1 D_\rho\times_2 A_\theta\times_3 D_\phi+
  \boldsymbol W_n\times_1 D_\rho\times_3 A_\phi+
  \boldsymbol G_n,\\
  \boldsymbol W_{n+1}&=\boldsymbol W_n+\tau
  \Bigg(\bigg(\boldsymbol \Phi_{\rho,\theta,\phi}\circ\Big(
 \big((\boldsymbol \Phi_{\rho,\cdot,\phi}\circ(\boldsymbol F_n\times_3 V_\phi^{-1}))\times_3 V_\phi\big)
  \times_2 Q_\theta^{\sf T}\Big)\bigg)\times_2 Q_\theta\Bigg)
    \times_1 \varphi_1(\tau A_{\rho,3}),
  \end{aligned}\right.
\end{equation}
where we remind that $V_\phi^{-1}=Q_\phi^{\sf T}\Xi^{-1}$
(see the proof of Proposition~\ref{prop:tridiag}).

We conclude by considering the ODEs system
    for the cylinder case  \eqref{eq:cylinder}.
First of all, the action of the matrix $M$ is computed as
\begin{equation*}
\boldsymbol W\times_1 A_{\rho,2}+
  \boldsymbol W\times_1 D_\rho\times_2 A_\theta+
  \boldsymbol W\times_3 A_z.
\end{equation*}
The action $\varphi_1(\tau M_3)\boldsymbol w$ is realized by
\begin{equation*}
\boldsymbol W\times_3 \varphi_1(\tau A_z),
\end{equation*}
 while the action $\varphi_1(\tau M_2)\boldsymbol w$ is computed by
 \begin{equation*}
   \big(\boldsymbol \Phi_{\rho,\theta,\cdot}\circ
   (\boldsymbol W\times_2 Q_\theta^{\sf T})\big)\times_2 Q_\theta.
\end{equation*}
Here $\boldsymbol \Phi_{\rho,\theta,\cdot}$ is the order-3 tensor with element
$\varphi_1(\tau\rho_i^{-2}{\lambda_\theta}_j)$ in position $(i,j,k)$.
The action of $\varphi_1(\tau M_1)\boldsymbol w$ is realized by
\begin{equation*}
\boldsymbol W\times_1 \varphi_1(\tau A_{\rho,2}).
\end{equation*}
Notice that both the matrices $\varphi_1(\tau A_z)$ and $\varphi_1(\tau A_{\rho,2})$
can be conveniently computed by diagonalization, since the relevant
transformation matrices are (weakly) well conditioned.
Hence, the split exponential Euler method for the cylinder
case~\eqref{eq:cylinder} is realized in tensor form by
\begin{equation}\label{eq:ee_cylinder}
\left\{
\begin{aligned}
  \boldsymbol F_n&=\boldsymbol W_n\times_1 A_{\rho,2}+
  \boldsymbol W_n\times_1 D_\rho\times_2 A_\theta+
  \boldsymbol W_n\times_3 A_z+
  \boldsymbol G_n,\\
\boldsymbol W_{n+1}&=\boldsymbol W_n+\tau \bigg(\Big(\boldsymbol \Phi_{\rho,\theta,\cdot}\circ\big((\boldsymbol F_n\times_3 \varphi_1(\tau A_z))
  \times_2 Q_\theta^{\sf T}\big)\Big)\times_2 Q_\theta\bigg)\times_1 \varphi_1(\tau A_{\rho,2}).
  \end{aligned}\right.
\end{equation}
\end{subequations}
 
We highlight that the employment
of the proposed matrix- and tensor-oriented techniques leads to the
direct schemes~\eqref{eq:sexpe}, where
the relevant quantities are computed by high-performance level 3 BLAS.
In particular, no iterative
method (e.g., Krylov subspace methods) is needed for
the approximation of the matrix functions, nor 
we require the choice and/or tuning of input tolerances or hyperparameters.
Hence,
the computational cost of the proposed approach
is constant per time step.
Finally, we stress that the computational burden of the diagonalizations and scalar computations
of the $\varphi_1$ functions is negligible compared to that of the time integration.
In fact, we perform these calculations on small-sized quantities once and for
all before the actual start of the time marching.

\section{Numerical experiments}\label{sec:numer_exp}
In this section, we validate the proposed matrix- and tensor-oriented 
procedures on a variety of two- and three-dimensional examples with different
geometries for the spatial domains.
As already mentioned in the introduction, we consider as test problems
\textit{systems} of {\color{black} $N$-coupled}
diffusion--reaction equations leading to Turing patterns,
that are of great interest for their effective application in many areas.
Despite not being a \textit{single} PDE as
in formula~\eqref{eq:pde}, in these systems the differential operators act
separately on the involved unknowns, and the coupling is realized by the nonlinear
parts. Therefore, the proposed techniques can be easily adapted to handle the
equations under study (see the following section for more insights and, e.g.,
\cite{CC24bis,C24}).

All the numerical experiments are performed on a standard laptop equipped with an
Intel\textsuperscript{\textregistered} Core\textsuperscript{\texttrademark}
i7-10750H chip (six physical cores) and 16GB of RAM. We employ
MathWorks MATLAB\textsuperscript{\textregistered} R2025a as a software.
The {\color{black}open} source code to reproduce all the experiments, fully compatible with
GNU Octave, can be found in a {\color{black} public} GitHub
repository\footnote{Available at \url{https://github.com/cassinif/splitEE_curvilinear}.}.

To realize the required matrix-tensor operations, we use the relevant functions
in KronPACK\footnote{Available at \url{https://github.com/caliarim/KronPACK}.}
(see~\cite{CCZ23bis}).
In addition, the needed eigenpairs are computed using the
built-in MATLAB function \texttt{eig}.
Finally, the initial conditions for the models are small random perturbations
of steady states of the systems (as common in the context of Turing patterns).
To this aim, for reproducibility of the results we fix in each example
the MATLAB random generation seed with the command \texttt{rng(2)}.

\subsection{{\color{black}Comparison of the proposed method with other
    integrators}}\label{sec:orderex}
In the first numerical example we demonstrate the expected convergence of
the split exponential Euler method and highlight the performance of the proposed
technique in comparison to other numerical integrators. To this aim,
we consider
\begin{subequations}\label{eq:BVAM}
\begin{equation}\label{eq:BVAMsys}
  \left\{
    \begin{aligned}
      \partial_t u &= \gamma \Delta_\mathrm{D} u + b(u,v)
              {\color{black}+b_\mathrm{s}}&& \text{on } \Omega,\\
              \partial_t v &= \delta \Delta_\mathrm{D} v + c(u,v)
              {\color{black}+c_\mathrm{s}}&& \text{on } \Omega,\\
    \nabla u\cdot \mathbf{n}&=0, \quad \nabla v\cdot \mathbf{n}=0
    &&\text{on } \partial\Omega,
    \end{aligned}
  \right.
\end{equation}
with
\begin{equation}\label{eq:BVAMnl}
  \begin{aligned}
  b(u,v) &= \alpha_1 u (1-\alpha_2v^2) + v(1-\alpha_3u), \\
  c(u,v) &= \beta_1 v \left( 1+\frac{\alpha_1 \alpha_2}{\beta_1}uv\right)
           + u(\beta_2 + \alpha_3 v),
  \end{aligned}
\end{equation}
\end{subequations}
on a unitary disk (radius $\rho_*=1$){\color{black}, which is
  the so-called BVAM model~\cite{ATGBM02} with the additional sources
  $b_\mathrm{s}=b_\mathrm{s}(t,\rho,\theta)$
  and $c_\mathrm{s}=c_\mathrm{s}(t,\rho,\theta)$ such that the exact solution
  is
  \begin{equation*}
      u=\rme^{-t}(1-\rho)^2\rme^{\frac{\cos\theta}{10}},\quad
      v=\rme^{-t}(1-\rho)^2\rme^{\frac{\sin\theta}{10}}.
    \end{equation*}
The symbol $\Delta_{\mathrm{D}}$ denotes the Laplace operator
on the disk defined in formula~\eqref{eq:lapdisk}. The parameters are set to $\gamma=3.87\cdot10^{-3}$,
$\alpha_1=-\beta_2=8.99\cdot10^{-1}$, $\alpha_2=\alpha_3=2\cdot10^{-1}$,
$\delta=7.5\cdot10^{-3}$, and $\beta_1=-9.1\cdot10^{-1}$.
The initial condition is
\begin{equation*}
  u_0=(1-\rho)^2\rme^{\frac{\cos\theta}{10}},\quad
  v_0=(1-\rho)^2\rme^{\frac{\sin\theta}{10}}.
\end{equation*}%
For the convenience of the reader,} we explicitly write
the system of ODEs arising from the semidiscretization of the coupled
PDEs~\eqref{eq:BVAM}. In fact, in vector form we get
\begin{equation}\label{eq:BVAM_ODE_vec}
  \left\{\begin{aligned}
  \boldsymbol u'(t)&=\gamma(I\otimes A_{\rho,2}+
  A_\theta\otimes D_\rho)  \boldsymbol u(t)+
  \boldsymbol b({\color{black}t},\boldsymbol u(t),\boldsymbol v(t)),\\
  \boldsymbol v'(t)&=\delta(I\otimes A_{\rho,2}+
  A_\theta\otimes D_\rho) \boldsymbol v(t)+
  \boldsymbol c({\color{black}t},\boldsymbol u(t),\boldsymbol v(t)),\\
  \boldsymbol u(0)&=    \boldsymbol u_0,\\
  \boldsymbol v(0)&=  \boldsymbol v_0,
\end{aligned}
\right.
\end{equation}
where the relevant matrices have been introduced in Section~\ref{sec:space_disc}.
Then, the previous system can be written in matrix form as
\begin{equation*}
  \left\{\begin{aligned}
  \boldsymbol U'(t)&=\gamma (A_{\rho,2}\boldsymbol U(t)+
  D_\rho \boldsymbol U(t)A_\theta)+
  \boldsymbol B({\color{black}t},\boldsymbol U(t),\boldsymbol V(t)),\\
  \boldsymbol V'(t)&=\delta( A_{\rho,2}\boldsymbol V(t)+
  D_\rho \boldsymbol V(t)A_\theta)+
  \boldsymbol C({\color{black}t},\boldsymbol U(t),\boldsymbol V(t)),\\
  \boldsymbol U(0)&=  \boldsymbol U_0,\\
  \boldsymbol V(0)&=\boldsymbol V_0,
\end{aligned}
\right.
\end{equation*}
and the split exponential Euler scheme results in
\begin{equation}\label{eq:BVAM_SExpE}
  \left\{\begin{aligned}
  \boldsymbol E_{n}&=\gamma( A_{\rho,2}\boldsymbol U_n+
  D_\rho\boldsymbol U_nA_\theta)+
\boldsymbol B_n,\\
  \boldsymbol U_{n+1}&=\boldsymbol U_{n}+\tau
\varphi_1(\tau\gamma A_{\rho,2})(\boldsymbol \Phi
\circ(\boldsymbol E_{n}Q_\theta))Q_\theta^{\sf T},\\
  \boldsymbol F_{n}&=\delta( A_{\rho,2}\boldsymbol V_n+
  D_\rho\boldsymbol V_nA_\theta)+
\boldsymbol C_n,\\
  \boldsymbol V_{n+1}&=\boldsymbol V_{n}+\tau
\varphi_1(\tau\delta A_{\rho,2})(\boldsymbol \Psi
\circ(\boldsymbol F_{n}Q_\theta))Q_\theta^{\sf T},
\end{aligned}\right.
  \end{equation}
cf.~with formula~\eqref{eq:ee_disk}.
Here the order-2 tensors $\boldsymbol B_n$ and $\boldsymbol C_n$ collect
the values of $\boldsymbol b({\color{black}t_n},\boldsymbol u_n,\boldsymbol v_n)$
and $\boldsymbol c({\color{black}t_n},\boldsymbol u_n,\boldsymbol v_n)$,
respectively, and $\boldsymbol \Phi$ and
$\boldsymbol \Psi$ are the order-2 tensors
with elements $\varphi_1(\tau\gamma\rho_i^{-2}{\lambda_{\theta}}_j)$ and
$\varphi_1(\tau\delta\rho_i^{-2}{\lambda_{\theta}}_j)$, respectively.

We will compare the performance of the proposed scheme against the classical
exponential Euler method~\eqref{eq:ee}
applied to system~\eqref{eq:BVAM_ODE_vec}. In fact,
it can be written as
\begin{equation}\label{eq:BVAM_ExpE}
  \left\{\begin{aligned}
  \boldsymbol e_{n}&=\gamma (I\otimes A_{\rho,2}+
  A_\theta\otimes D_\rho)\boldsymbol u_n+
\boldsymbol b_n,\\
  \boldsymbol u_{n+1}&=\boldsymbol u_{n}+\tau
\varphi_1(\tau\gamma (I\otimes A_{\rho,2}+
A_\theta\otimes D_\rho))\boldsymbol e_{n},\\
  \boldsymbol f_{n}&=\delta (I\otimes A_{\rho,2}+
  A_\theta\otimes D_\rho)\boldsymbol v_n+
\boldsymbol c_n,\\
  \boldsymbol v_{n+1}&=\boldsymbol v_{n}+\tau
\varphi_1(\tau\delta (I\otimes A_{\rho,2}+
A_\theta\otimes D_\rho))\boldsymbol f_{n}.
\end{aligned}\right.
\end{equation}
In practice, the needed actions of $\varphi_1$ functions are realized using
a state-of-the-art Krylov solver with incomplete orthogonalization,
namely
KIOPS\footnote{Available at \url{https://gitlab.com/stephane.gaudreault/kiops}.}~\cite{GRT18}.
Other suitable techniques for the computation of
the action of the $\varphi_1$ function of large and sparse matrices
are described in~\cite{AMH11,CCZ23,LPR19}.
Furthermore, we compare the performance against the classical forward Euler
method applied to system~\eqref{eq:BVAM_ODE_vec}. Despite the well-known
limitations of the scheme in the stiff context, it is widely employed in the
literature for Turing patterns
(see, e.g., \cite{ATGBM02,SMTT23,VAB99}) because of its
simplicity of implementation and overall good performance when the number of
degrees of freedom, and thus the stiffness, is moderate.

We measure the performance of the schemes in terms of wall-clock time
(reported in seconds in the plots) and error at the final time
{\color{black}$t_*=1$}.
For the latter, we consider the 2-norm of the relative {\color{black}$L^2$ norm}
of the
two components of the system. We report in Figure~\ref{fig:plot_order} the
outcome of the simulations for increasing number of time steps $m$ and
degrees of freedom (DOF) $(n_\rho,n_\theta)=(25,50)$ (top plots)
and $(n_\rho,n_\theta)=(40,80)$ (bottom plots).
First of all, notice that all the methods show the expected first-order
convergence.
Then, as anticipated, the forward Euler method has a time step size
restriction due to
the lack of favorable stability properties in the stiff context.
As a matter of fact, it is not able to produce an approximation
in the case with lower DOF unless a sufficiently large number
of time steps is employed. Increasing the DOF (bottom plots) always
results in a divergent behavior for the chosen range of $m$. Also,
when the forward Euler method converges, it has a slightly
higher error compared
to that of the exponential integrators. These downfalls, together with the
computational cost comparable to that of the proposed split exponential Euler
method, allow us to conclude that the forward Euler method should not be
considered as a robust approach in our context.
Considering more in detail the performance of the exponential integrators, we
observe that no time step size restriction is required to guarantee stability.
Moreover, we see that the error of the split
exponential Euler method is comparable to that of exponential Euler. This is
expected, since we are inserting a second-order approximation for the
$\varphi_1$ function in a first-order integrator,
see formula~\eqref{eq:splitorder}.
On the other hand, in terms of computational cost, the advantage of employing
the split version of exponential Euler is clear. In fact, in average it is
17 times faster than the (classical) exponential Euler method. The superiority
of matrix- and tensor-oriented techniques over vector approaches in square and
cubic domains was already
observed, e.g., in~\cite{CC24bis,C24}.
Remark that the split exponential Euler method has a constant cost per time step
(roughly $0.02$ and $0.1$ milliseconds for the top and bottom plots,
respectively)
since it is a \textit{direct} method. Moreover, as mentioned in a previous
section, it has an initialization phase of negligible computational burden
compared to that of the overall time integration. In this phase,
we compute the quantities required by the split scheme
(that is, eigenpairs and $\varphi_1$ functions of scalar arguments), and in fact
it takes at most $0.2$ milliseconds for both experiments.
In summary, all the above considerations
allow us to conclude
that the split exponential Euler method is the preferred method in our
setting. For brevity of exposition, in the following experiments we will
present only the results with this scheme.
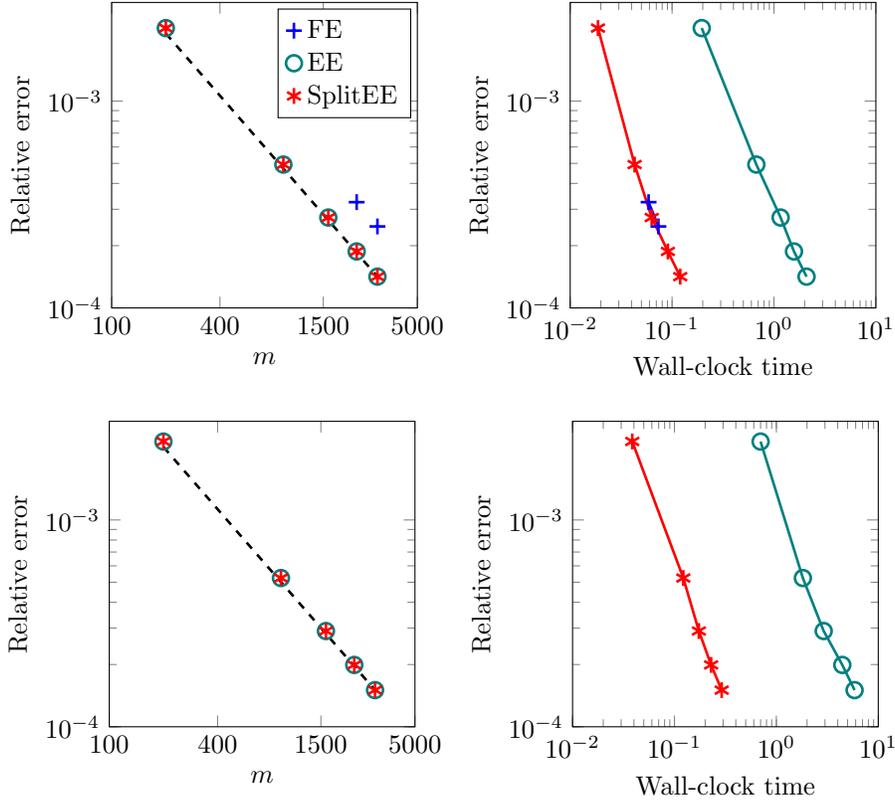
\begin{figure}[htb]
  \centering
%
%
%
\begin{tikzpicture}

\begin{axis}[%
width=1.6in,
height=1.6in,
at={(-0.4in,0in)},
scale only axis,
xmode=log,
xmin=100,
xmax=5000,
xtick={100,400,1500,5000},
xticklabels={100,400,1500,5000},
xminorticks=true,
xlabel={$m$},
ymode=log,
ymin=3e-4,
ymax=1e-2,
yminorticks=true,
ylabel={Relative error},
axis background/.style={fill=white},
legend style={legend cell align=left, align=left}
]
\addplot [color=blue, only marks,line width=1pt, mark size = 3pt, mark=square, mark options={solid, blue}]
  table[row sep=crcr]{%
2300	7.2421e-04\\
3000	5.5611e-04\\
};
\addlegendentry{FE}

\addplot [color=teal, only marks,line width=1pt, mark size = 3pt, mark=o, mark options={solid, teal}]
  table[row sep=crcr]{%
200	7.9801e-03\\
900	1.7752e-03\\
1600	9.9944e-04\\
2300	6.9630e-04\\
3000	5.3496e-04\\
};
\addlegendentry{EE}

\addplot [color=red, only marks,line width=1pt, mark size = 3pt, mark=asterisk, mark options={solid, red}]
  table[row sep=crcr]{%
200	7.9801e-03\\
900	1.7752e-03\\
1600	9.9944e-04\\
2300	6.9630e-04\\
3000	5.3496e-04\\
};
\addlegendentry{SplitEE}

\addplot [color=black,line width=1pt, dashed, forget plot]
  table[row sep=crcr]{%
200	7.9801e-03\\
3000	5.3496e-04\\
};
\end{axis}

\begin{axis}[%
width=1.6in,
height=1.6in,
at={(2in,0in)},
scale only axis,
xmode=log,
xmin=0.01,
xmax=10,
xminorticks=true,
xlabel={Wall-clock time},
ymode=log,
ymin=3e-4,
ymax=1e-2,
yminorticks=true,
ylabel={Relative error},
axis background/.style={fill=white},
legend style={legend cell align=left, align=left}
]
\addplot [color=blue, mark=square,line width=1pt, mark size = 3pt, mark options={solid, blue}]
  table[row sep=crcr]{%
0.079798	7.2421e-04\\
0.10631	5.5611e-04\\
};

\addplot [color=teal, mark=o,line width=1pt, mark size = 3pt, mark options={solid, teal}]
  table[row sep=crcr]{%
0.195493	7.9801e-03\\
0.671324	1.7752e-03\\
1.161818	9.9944e-04\\
1.572234	6.9630e-04\\
2.087686	5.3496e-04\\
};

\addplot [color=red, mark=asterisk,line width=1pt, mark size = 3pt, mark options={solid, red}]
  table[row sep=crcr]{%
0.018746	7.9801e-03\\
0.042881	1.7752e-03\\
0.063145	9.9944e-04\\
0.090922	6.9630e-04\\
0.120168	5.3496e-04\\
};

\end{axis}

\end{tikzpicture}%
\\[3ex]
%
%
%
\begin{tikzpicture}

\begin{axis}[%
width=1.6in,
height=1.6in,
at={(-0.425in,0in)},
scale only axis,
xmode=log,
xmin=100,
xmax=5000,
xtick={100,400,1500,5000},
xticklabels={100,400,1500,5000},
xminorticks=true,
xlabel={$m$},
ymode=log,
ymin=3e-4,
ymax=1e-2,
yminorticks=true,
ylabel={Relative error},
axis background/.style={fill=white},
legend style={legend cell align=left, align=left}
]

\addplot [color=teal, only marks,line width=1pt, mark size = 3pt, mark=o, mark options={solid, teal}]
  table[row sep=crcr]{%
200	7.9290e-03\\
900	1.7734e-03\\
1600	9.9771e-04\\
2300	6.9419e-04\\
3000	5.3235e-04\\
};

\addplot [color=red, only marks,line width=1pt, mark size = 3pt, mark=asterisk, mark options={solid, red}]
  table[row sep=crcr]{%
200	7.9741e-03\\
900	1.7734e-03\\
1600	9.9771e-04\\
2300	6.9419e-04\\
3000	5.3235e-04\\
};

\addplot [color=black,line width=1pt, dashed, forget plot]
  table[row sep=crcr]{%
200	7.9741e-03\\
3000	5.3235e-04\\
};
\end{axis}

\begin{axis}[%
width=1.6in,
height=1.6in,
at={(2in,0in)},
scale only axis,
xmode=log,
xmin=0.01,
xmax=10,
xminorticks=true,
xlabel={Wall-clock time},
ymode=log,
ymin=3e-4,
ymax=1e-2,
yminorticks=true,
ylabel={Relative error},
axis background/.style={fill=white},
legend style={legend cell align=left, align=left}
]

\addplot [color=teal, mark=o,line width=1pt, mark size = 3pt, mark options={solid, teal}]
  table[row sep=crcr]{%
0.70022	7.9290e-03\\
1.82061	1.7734e-03\\
2.91355	9.9771e-04\\
4.43852	6.9419e-04\\
5.85585	5.3235e-04\\
};

\addplot [color=red, mark=asterisk,line width=1pt, mark size = 3pt, mark options={solid, red}]
  table[row sep=crcr]{%
0.038543	7.9741e-03\\
0.122089	1.7734e-03\\
0.173542	9.9771e-04\\
0.228798	6.9419e-04\\
0.290617	5.3235e-04\\
};

\end{axis}

\end{tikzpicture}%
  \caption{Results of the experiment in Section~\ref{sec:orderex} for varying number
  of time steps $m$. The DOF in the top plots are
  $(n_\rho,n_\theta)=(25,50)$, while in the bottom plots are
  $(n_\rho,n_\theta)=(40,80)$.
  The reference dashed line has slope $-1$. In the legend, FE denotes the 
  forward Euler method, EE the exponential Euler scheme~\eqref{eq:BVAM_ExpE},
  and SplitEE the proposed split
  exponential Euler integrator~\eqref{eq:BVAM_SExpE}.}
  \label{fig:plot_order}
\end{figure}

\subsection{BVAM model on the disk}\label{sec:disk}
We now consider {\color{black} the BVAM model, namely
system~\eqref{eq:BVAM} without the source terms $b_\mathrm{s}$ and
$c_\mathrm{s}$,
taking the final time $t_* = 800$, see~\cite{ATGBM02}}.
The choice of the parameters triggers Turing instability of the
equilibrium $(u_\mathrm{e},v_\mathrm{e})=(0,0)$, eventually leading to a static
but spatially inhomogeneous state.
{\color{black} The initial condition is
\begin{equation*}
  u_0 = 10^{-2}\cdot\mathcal{N}(0,1), \quad v_0 = 10^{-2}\cdot\mathcal{N}(0,1).
\end{equation*}
Here and in the following $\mathcal{N}$ denotes the normal random variable.}
The DOF per direction
are set to $(n_\rho,n_\theta)=(40,80)$, and the time-marching integrator
is again~\eqref{eq:BVAM_SExpE} with {\color{black}$m=5000$} time steps.

The outcome of the simulation is presented in
Figure~\ref{fig:plot_turing_disk_BVAM}. The results are
in perfect agreement with what already reported in the literature
(see~\cite[Fig.~2(b)]{ATGBM02}), namely the resulting pattern in
the $u$ component is a spot pattern with symmetry. The wall-clock time of the
simulation is approximately {\color{black}0.5 seconds}. We also illustrate
here (and in the following experiments)
the evolution
of {\color{black}a discretization of the integral mean of $u$
\begin{equation*}
 \frac{1}{\lvert \Omega \rvert}\int_\Omega u(t_n,x,y)dxdy,
\end{equation*}
and of the time increment of the solution in the
      $L^2$ norm
      \begin{equation*}
        \left(\int_{\Omega}(u(t_{n+1},x,y)-u(t_n,x,y))^2dx dy\right)^{\frac{1}{2}}
      \end{equation*}
as indicators} of the system reaching the desired state. As expected,
after the initial reactivity phase, the integral mean stabilizes and remains constant,
{\color{black}while the increment decreases to a value a few orders of magnitude
 less than the peak.}
\begin{figure}[htb]
  \centering
  \includegraphics[scale=0.165]{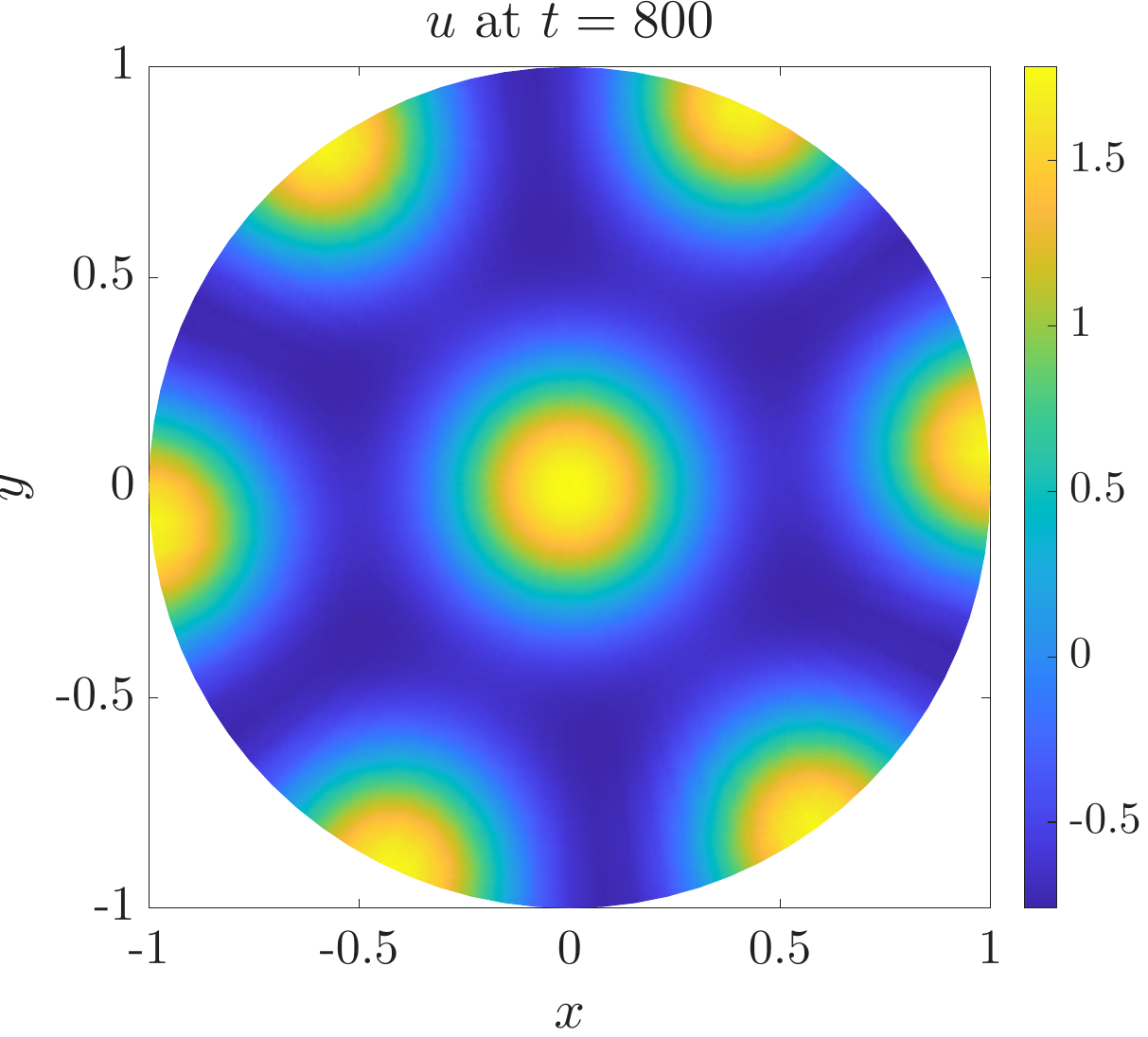}
  \includegraphics[scale=0.165]{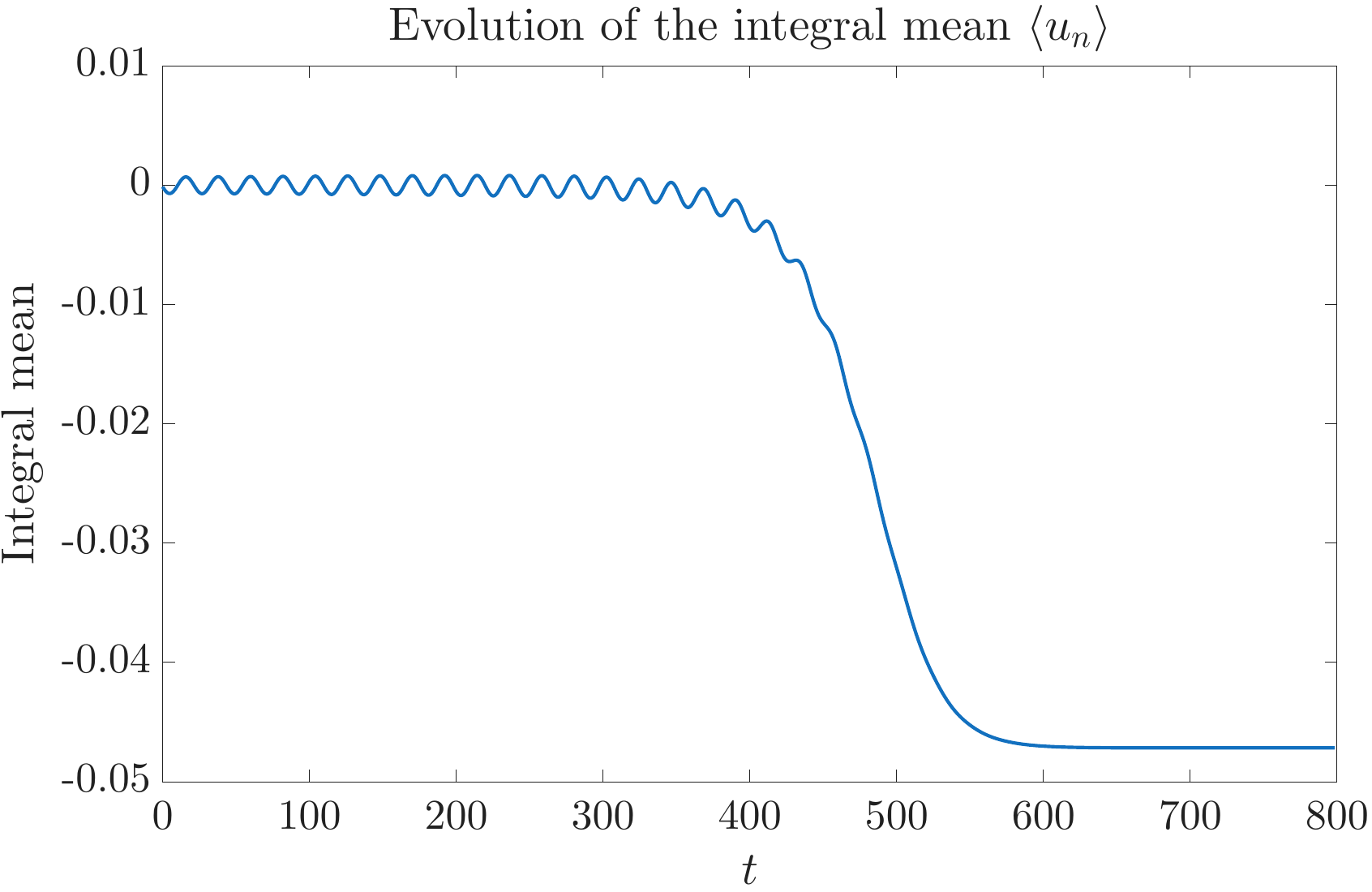}
  \includegraphics[scale=0.165]{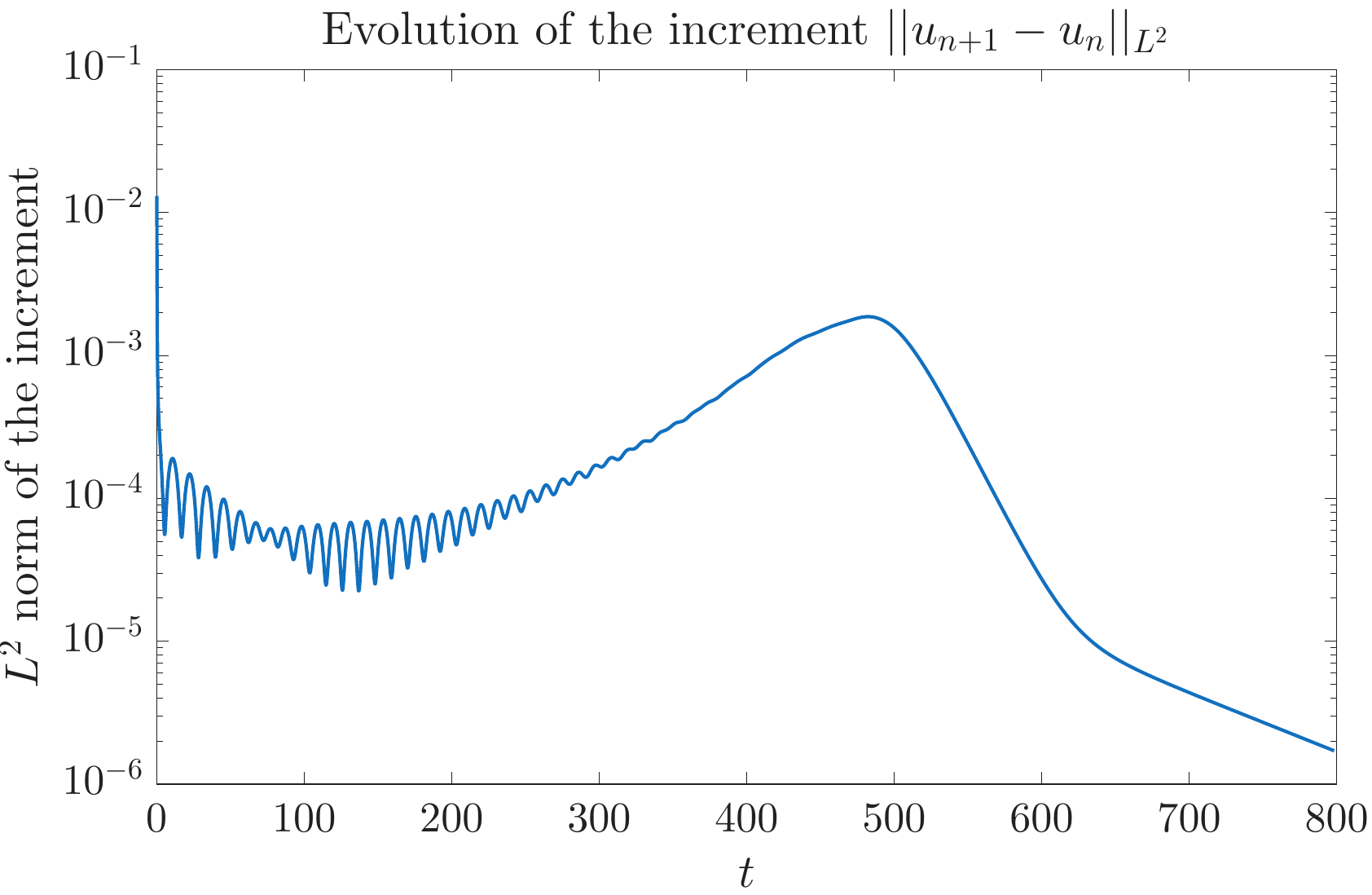}
  \caption{Results of the experiment in Section~\ref{sec:disk} using
  integrator~\eqref{eq:BVAM_SExpE}. The DOF per
  direction are $(n_\rho,n_\theta)=(40,80)$ (total number of DOF $2\cdot 3200$),
  and the number of time steps is {\color{black}$m=5000$. Left: $u$ component
    at the final time. Center: evolution of the integral mean.
    Right: evolution of the increment.
  The total simulation wall-clock time is about 0.5  seconds.}}
  \label{fig:plot_turing_disk_BVAM}
\end{figure}

\subsection{Schnakenberg model with anomalous diffusion on the disk}\label{sec:anomalous}
We now consider a different coupled system, namely
the following \textit{anomalous} diffusion--reaction model with Schnakenberg
kinetics~\cite{LHAMMHH24}
\begin{subequations}\label{eq:schnak}
\begin{equation}\label{eq:schnaksys}
  \left\{
    \begin{aligned}
    \partial_t u &= \Delta_\lambda u + \alpha_1 b(u,v) && \text{on } \Omega,\\
    \partial_t v &= \delta \Delta_\lambda v + \alpha_1 c(u,v)
    && \text{on } \Omega,\\
    u&={\color{black} \alpha_2+\beta_1}, &&\text{on } \partial\Omega,\\
    v&={\color{black}\frac{\beta_1}{(\alpha_2+\beta_1)^2}}
    &&\text{on } \partial\Omega,
    \end{aligned}
  \right.
\end{equation}
where
\begin{equation}\label{eq:schnaknl}
  \begin{aligned}
  b(u,v) &= \alpha_2 - u + u^2v, \\
  c(u,v) &= \beta_1 - u^2v.
  \end{aligned}
\end{equation}
\end{subequations}
The spatial domain $\Omega$ is the unitary disk (that is, the radius is 
$\rho_*=1$).
The anomalous diffusion operator $\Delta_\lambda$ is defined by
\begin{equation}\label{eq:anomalous}
  \Delta_\lambda = \left(
  \frac{1-\lambda}{\rho^{1+\lambda}}\frac{\partial}{\partial \rho}+
  \frac{1}{\rho^\lambda}\frac{\partial^2}{\partial \rho^2}\right)
  +\frac{1}{\rho^{2+\lambda}}\frac{\partial^2}{\partial \theta^2},
\end{equation}
see~\cite{LHAMMHH24,HHHNLHC17}.
Positive values of $\lambda$ describe subdiffusion models, while negative
ones superdiffusion.
Clearly, $\lambda=0$ gives the classical Laplace operator in polar coordinates
$\Delta_\mathrm{D}$, see formula~\eqref{eq:lapdisk}.
The inhomogeneous boundary values for $u$ and $v$ are set to
$(u_\mathrm{e},v_\mathrm{e})=(\alpha_2+\beta_1,\beta_1/(\alpha_2+\beta_1)^2)$,
which corresponds to the spatially uniform steady state of the system.
The remaining parameters are fixed to $\alpha_1=5\cdot10^{2}$, $\alpha_2=1.4\cdot10^{-1}$,
$\delta=5\cdot10^{1}$, and $\beta_1=1.34$. We perform our simulation in a 
superdiffusion regime, and in particular we choose the parameter
$\lambda=-1.95$.
The initial condition is set to
\begin{equation*}
  u_0 = u_\mathrm{e}+{\color{black}10^{-5}\cdot\mathcal{U}(0,1)}, \quad
  v_0 = v_\mathrm{e}+{\color{black}10^{-5}\cdot\mathcal{U}(0,1)}.
\end{equation*}
    {\color{black}Here and in the following
      $\mathcal{U}$ denotes the uniform random variable.}
We discretize the operator~\eqref{eq:anomalous} following a similar approach
to that described in Section~\ref{sec:space_disc} for the disk. In particular, after
lifting the system so that we have a model with homogeneous Dirichlet boundary
conditions, we end up with a computational grid in the $\rho$ coordinate and
corresponding tridiagonal discretization
matrix~\eqref{eq:tridiag} as
\begin{equation}\label{eq:A_lambda}
  \begin{aligned}
    \rho_i&=(1-\lambda)\frac{h_\rho}{2}+(i-1)h_\rho,
    \quad h_\rho=\frac{\rho_*}{n+\frac{1-\lambda}{2}},\\
    a_\ell&=-\frac{2}{(\frac{1-\lambda}{2}+\ell-1)^\lambda h_\rho^{2+\lambda}},\quad
  b_\ell=\frac{1-\lambda+\ell-1}{(\frac{1-\lambda}{2}+\ell-1)^{1+\lambda}h_\rho^{2+\lambda}},\quad
  c_\ell=\frac{\ell}{(\frac{1-\lambda}{2}+\ell)^{1+\lambda}h_\rho^{2+\lambda}}.
\end{aligned}
\end{equation}
Notice that also in this anomalous diffusion case we obtain an essentially
nonnegative matrix with negative eigenvalues.
For ease of reading, the estimate on the condition number of the related
transformation matrix is reported in Appendix~\ref{sec:A_lambda}.
Clearly, the discretization matrix for the $\theta$ coordinate is
still~\eqref{eq:A_theta}.
Concerning the time marching, we employ the integrator in
formula~\eqref{eq:BVAM_SExpE}
with suitable adaptations, namely $A_{\rho,2}$ is replaced by the matrix defined above
and $D_\rho$ by the diagonal matrix
with entries $\rho_i^{-2-\lambda}$.
\begin{figure}[htb]
  \centering
  \includegraphics[scale=0.165]{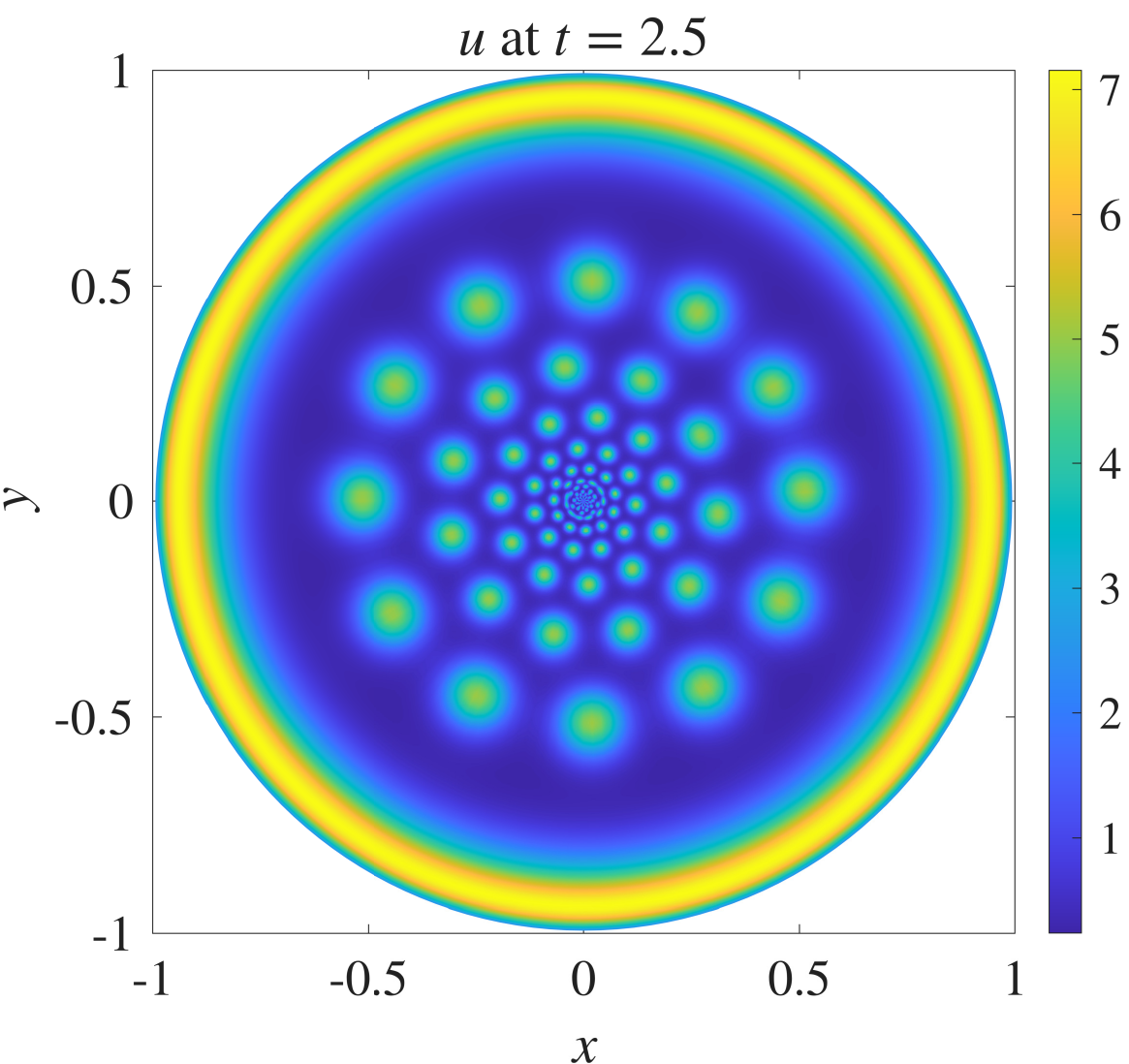}
  \includegraphics[scale=0.165]{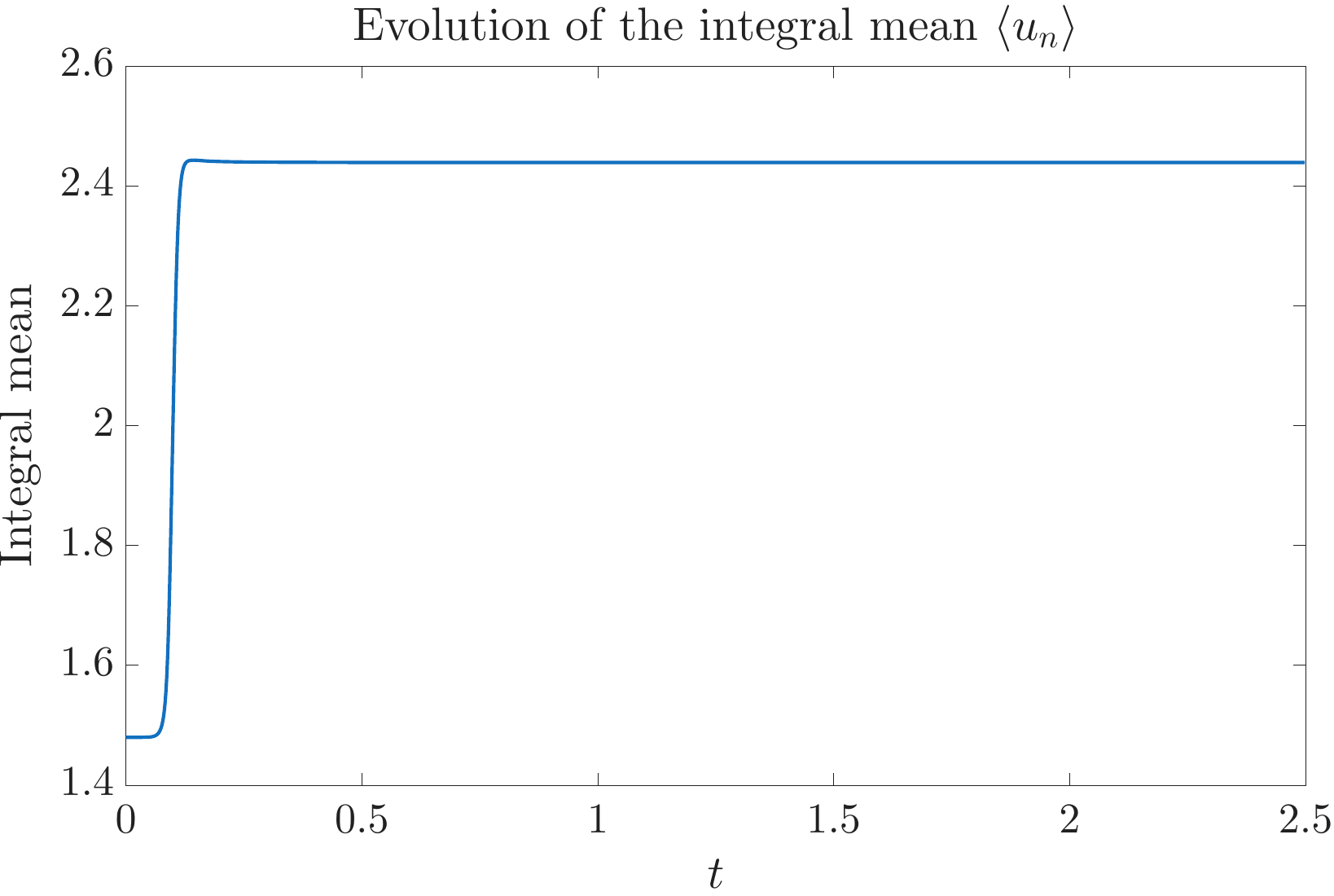}
 \includegraphics[scale=0.165]{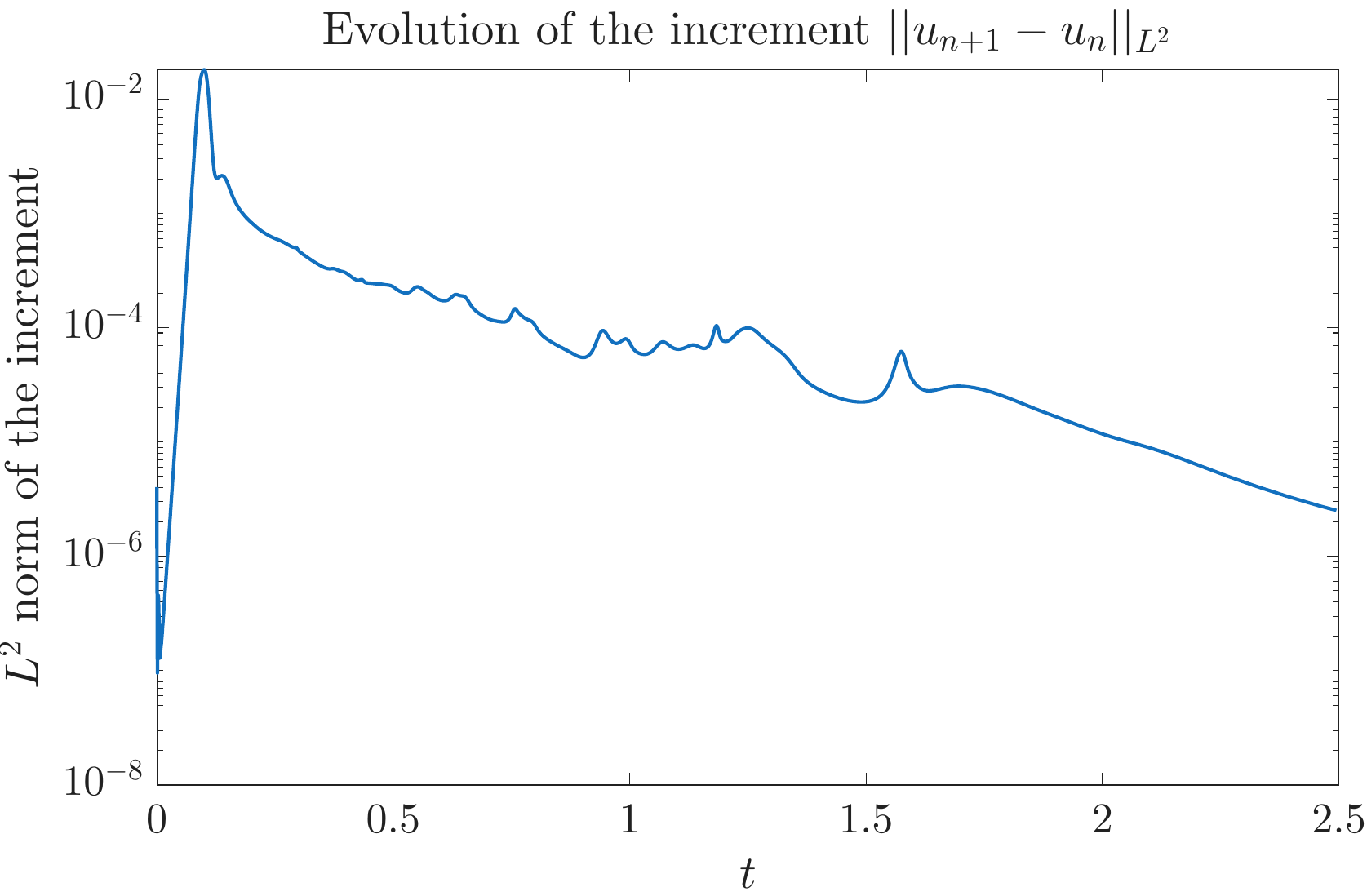}
  \caption{Results of the experiment in Section~\ref{sec:anomalous} using
  a suitable adaptation of integrator~\eqref{eq:BVAM_SExpE}. The DOF per
  direction are $(n_\rho,n_\theta)=(160,160)$ (total number of
  DOF $2\cdot 25600$),
  and the number of time steps is $m=25000$. {\color{black}Left: $u$ component
  at the final time. Center: evolution of the integral mean.
  Right: evolution of the increment.}
  The total simulation wall-clock time is about 19 seconds.}
  \label{fig:plot_turing_disk_schnakenberg_anomalous}
\end{figure}

The results of the simulation at the final time $t_*=2.5$,
setting the DOF per direction
to $(n_\rho,n_\theta)=(160,160)$ and the number of time steps to
$m=25000$,
are depicted in Figure~\ref{fig:plot_turing_disk_schnakenberg_anomalous}.
Also in this case, the outcome perfectly matches with experiments
already available in the literature
(see~\cite[Fig.~9 top right]{LHAMMHH24}) and the integral mean stabilizes
as time proceeds, {\color{black}while the increment decreases}.
The total wall-clock time of this
simulation is approximately 19 seconds.
\subsection{DIB model on the sphere}\label{sec:sphere}
We proceed our experiments by considering the
DIB model on a sphere. 
In particular, the governing equations are~\cite{LBFS17}
\begin{subequations}\label{eq:DIB}
\begin{equation}\label{eq:DIBsys}
  \left\{
    \begin{aligned}
    \partial_t r &= \frac{1}{\rho_*^2}\Delta_\mathrm{S} r + \zeta_1 p(r,s) && \text{on } \Omega, \\
    \partial_t s &= \frac{\varepsilon}{\rho_*^2} \Delta_\mathrm{S} s
    + \zeta_1 q(r,s) && \text{on } \Omega, \\
    \end{aligned}
  \right.
\end{equation}
where
\begin{equation}\label{eq:DIBnl}
  \begin{aligned}
  p(r,s) &=  \zeta_2(1-s)r-\zeta_3 r^3 - \zeta_4 (s-\zeta_5), \\
  q(r,s) &= \eta_1(1+\eta_2r)(1-s)(1-\eta_3(1-s))-\eta_4s(1+\eta_3 s)(1+\eta_5 r),
  \end{aligned}
\end{equation}
\end{subequations}
and $\eta_4 = \eta_1(1-\zeta_5)(1-\eta_3+\eta_3\zeta_5)/(\zeta_5(1+\eta_3\zeta_5))$.
The symbol $\Delta_\mathrm{S}$ denotes the Laplace--Beltrami
operator~\eqref{eq:lapsphere}.
The spatial domain $\Omega$ is a sphere with radius $\rho_*=1.1653$, while
the parameter values are set to $\zeta_1=\zeta_2=10$, $\zeta_3=1$,
{\color{black}$\zeta_4=44$},
$\zeta_5=0.5$, $\eta_1=5$, $\eta_2=2.5$, $\eta_3=0.2$, $\eta_5=1.5$,
and $\varepsilon=20$. The initial condition is a small random perturbation of
the spatially homogeneous equilibrium $(u_\mathrm{e},v_\mathrm{e})=(0,\zeta_5)$,
that is
\begin{equation*}
  u_0 = u_\mathrm{e}+10^{-6}\cdot\mathcal{N}(0,1), \quad
  v_0 = v_\mathrm{e}+10^{-6}\cdot\mathcal{N}(0,1).
\end{equation*}
For the simulation we consider as DOF per direction
$(n_\theta,n_\phi)=(100,50)$ and the time-marching scheme
is the adaptation of integrator~\eqref{eq:ee_sphere} to coupled systems.
The final time is set to {\color{black}$t_*=8$}
and the number of time steps is {\color{black}$m=4000$}.

The outcome is depicted in
Figure~\ref{fig:plot_turing_sphere_DIB}, and it
resembles similar experiments available in the literature
(see, e.g., \cite[Fig.~2 {\color{black}bottom left}]{LBFS17}).
In particular, the choice
of the parameters triggers the Turing instability and the resulting pattern in
the $u$ component are {\color{black}axisymmetric stripes.
  Moreover, the stationary indicators show the expected behaviors.}
The wall-clock time of the
simulation is approximately {\color{black}1.2 seconds}.
\begin{figure}[htb]
  \centering
  \includegraphics[scale=0.155]{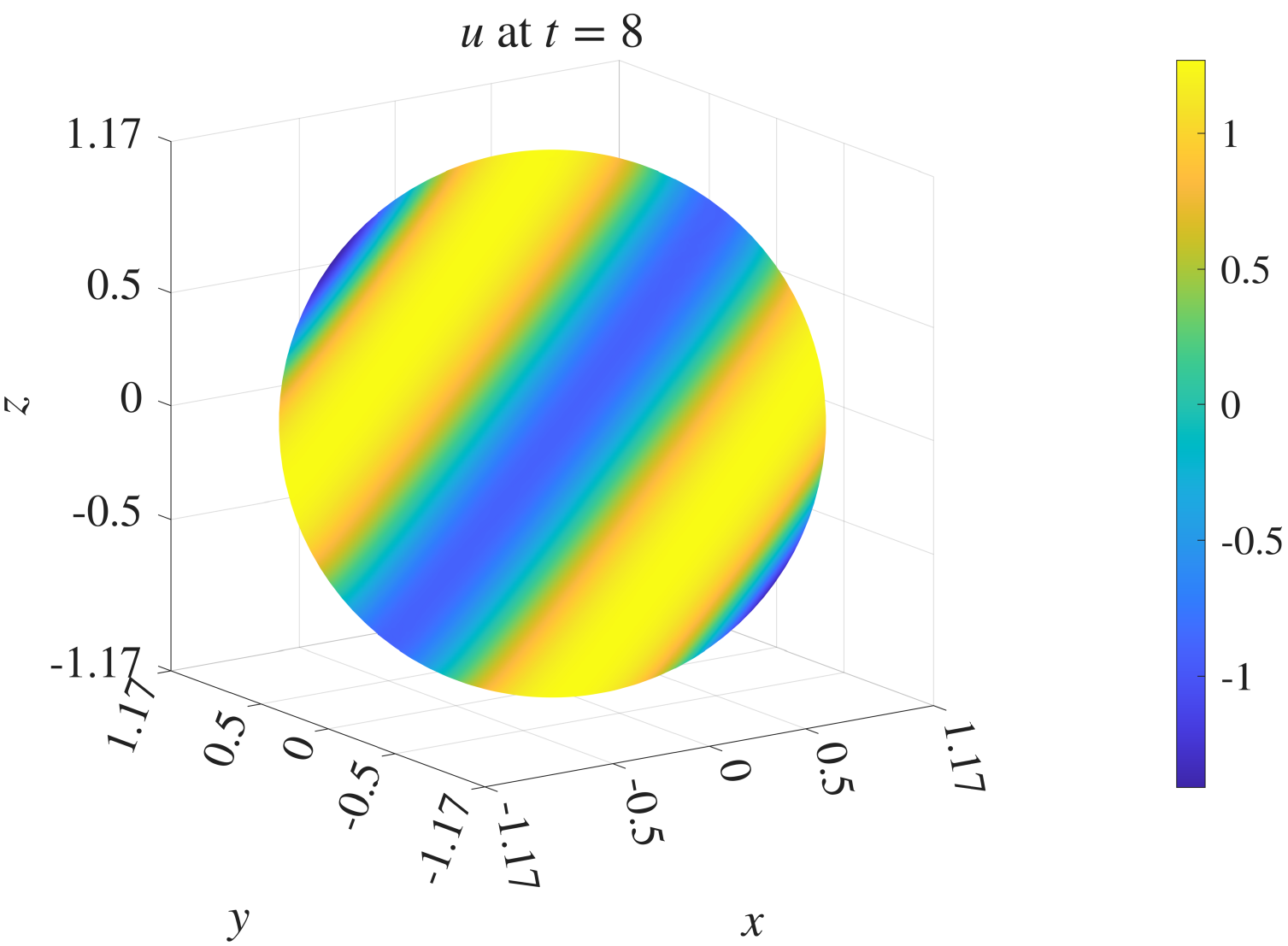}
  \includegraphics[scale=0.155]{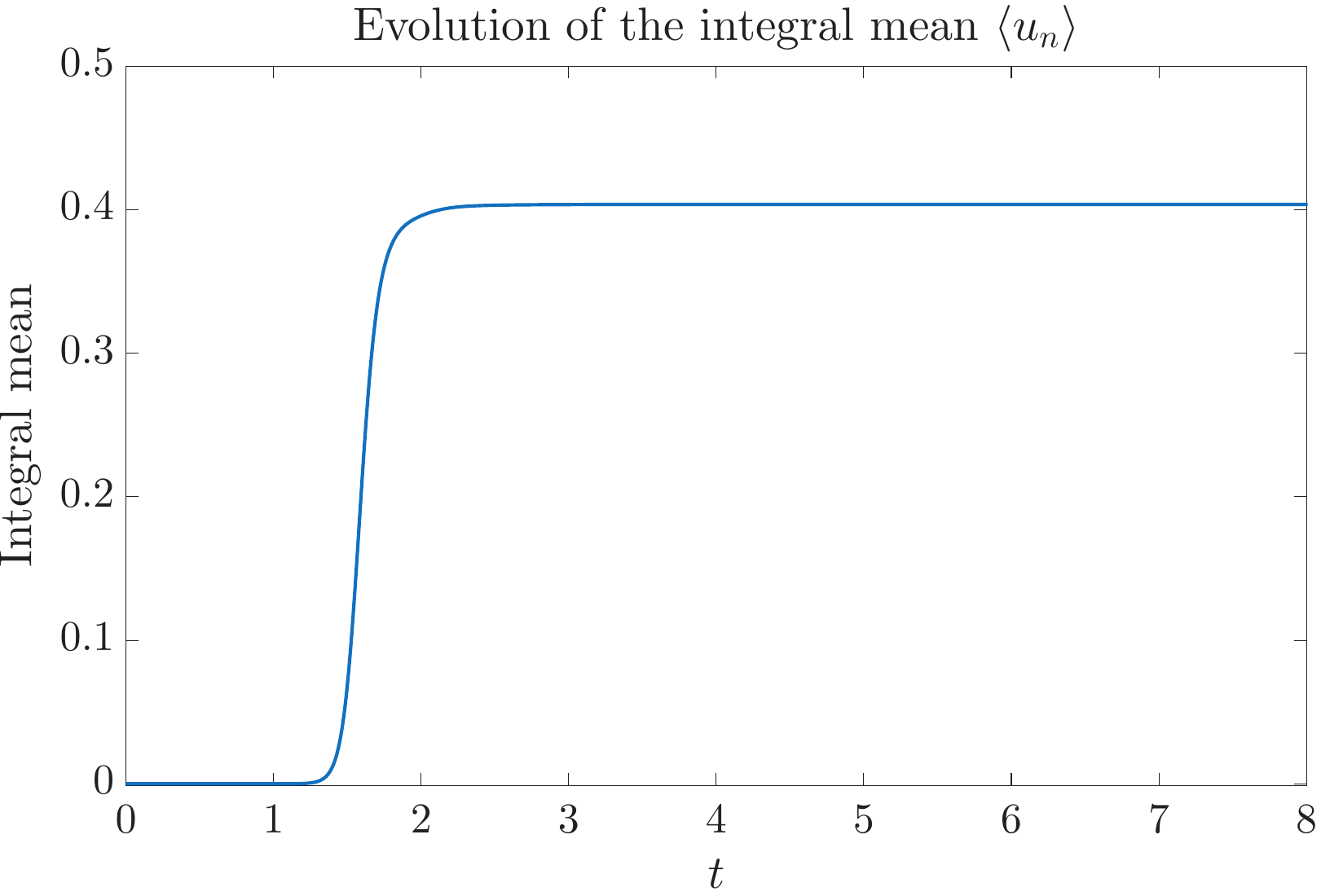}
    \includegraphics[scale=0.155]{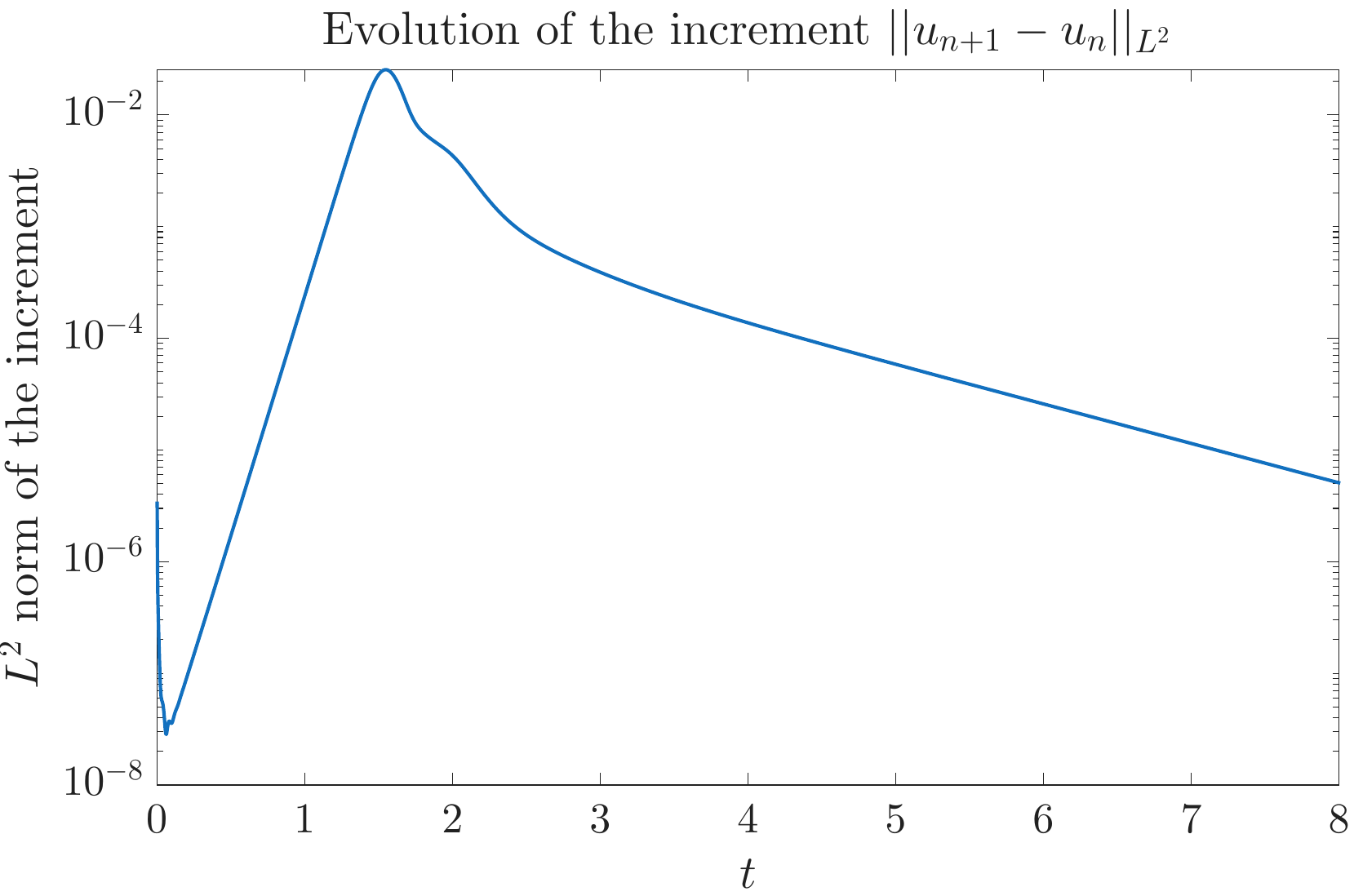}
  \caption{Results of the experiment in Section~\ref{sec:sphere} using
    the adaptation of integrator~\eqref{eq:ee_sphere} to coupled systems.
    The DOF per
    direction are $(n_\theta,n_\phi)=(100,50)$ (total number of
    DOF $2\cdot 5000$),
  and the number of time steps is {\color{black}$m=4000$. Left: $u$ component
    at the final time. Center: evolution of the integral mean. Right:
    evolution of
  the increment.
  The total simulation wall-clock time is about 1.2 seconds.}}
  \label{fig:plot_turing_sphere_DIB}
\end{figure}

\subsection{Bulk-surface Schnakenberg model on the sphere}\label{sec:ball}
We now compute a numerical approximation to the following four-components
bulk-surface Schnakenberg reaction--diffusion system on the sphere~\cite{FMS23}
\begin{subequations}\label{eq:bulkball}
\begin{equation}\label{eq:bulkballsys}
  \left\{
    \begin{aligned}
      \partial_t u &= \Delta_\mathrm{B} u +
      \alpha_1 b(u,v) && \text{on } \Omega, \\
    \partial_t v &= \delta \Delta_\mathrm{B} v
                   + \alpha_1 c(u,v) && \text{on } \Omega, \\
    \partial_t r &= \frac{1}{\rho_*^2}\Delta_\mathrm{S} r + \zeta_1 p(u,v,r,s) && \text{on } \partial\Omega,\\
    \partial_t s &= \frac{\varepsilon}{\rho_*^2} \Delta_\mathrm{S} s + \zeta_1 q(u,v,r,s) && \text{on } \partial\Omega,\\
    \nabla u\cdot\mathbf{n} &= \zeta_1(\zeta_2r-\zeta_3 u),\quad
    \delta\nabla v\cdot\mathbf{n} = \zeta_1(\eta_1s-\eta_2 v) &&
\text{on } \partial\Omega,
    \end{aligned}
  \right.
\end{equation}
where
\begin{equation}\label{eq:bulkballnl}
  \begin{aligned}
  b(u,v) &= \alpha_2 - u + u^2v,&p(u,v,r,s) &=  b(r,s) - (\zeta_2r-\zeta_3u), \\
  c(u,v) &= \beta_1 - u^2v,&  q(u,v,r,s) &= c(r,s) - (\eta_1s-\eta_2v).
  \end{aligned}
\end{equation}
\end{subequations}
The symbols $\Delta_\mathrm{S}$ and $\Delta_\mathrm{B}$ represent the
Laplace--Beltrami and the Laplace operators \eqref{eq:lapsphere} and
\eqref{eq:lapball}, respectively.
In the system, the first two equations are defined on the
spatial domain $\Omega$ which is a unitary ball (that is,
the radius is $\rho_*=1$).
The second and third equations define the dynamics at the boundary
surface (the unitary sphere). Finally, the last two
equations are the needed boundary conditions for $u$ and $v$.
In the model, we set the parameters to
$\alpha_1=\zeta_1=55$, $\alpha_2=0.1$, $\beta_1=0.9$, $\zeta_2=\zeta_3=5/12$,
$\eta_1=\eta_2=5$, and $\delta = \varepsilon = 10$.
The spatially homogeneous equilibrium
\begin{equation*}
(u_\mathrm{e},v_\mathrm{e},r_\mathrm{e},s_\mathrm{e})
=(\alpha_2+\beta_1,\beta_1/(\alpha_2+\beta_1)^2,\alpha_2+\beta_1,\beta_1/(\alpha_2+\beta_1)^2)
\end{equation*}
is subject to Turing instability. In fact, we take as the initial condition
\begin{equation*}
  \begin{aligned}
    u_0 &= u_\mathrm{e}+10^{-3}\cdot\mathcal{N}(0,1), &
    r_0 &= r_\mathrm{e}+10^{-3}\cdot\mathcal{N}(0,1),\\
    v_0 &= v_\mathrm{e}+10^{-3}\cdot\mathcal{N}(0,1), &
    s_0 &= s_\mathrm{e}+10^{-3}\cdot\mathcal{N}(0,1),
  \end{aligned}
\end{equation*}
and integrate the system up to the final time {\color{black}$t_*=15$} with
the adaptation of schemes~\eqref{eq:ee_sphere} and~\eqref{eq:ee_ball}
to coupled systems. Remark that, similarly
to the two-component cases presented previously, the linear
differential operators
act separately on the different components of the system, allowing us to employ
our proposed integrators after spatial semidiscretization (see the discussion
in Section~\ref{sec:orderex}). We omit the details for brevity of exposition.
\begin{figure}[!htb]
  \centering
  \includegraphics[scale=0.165]{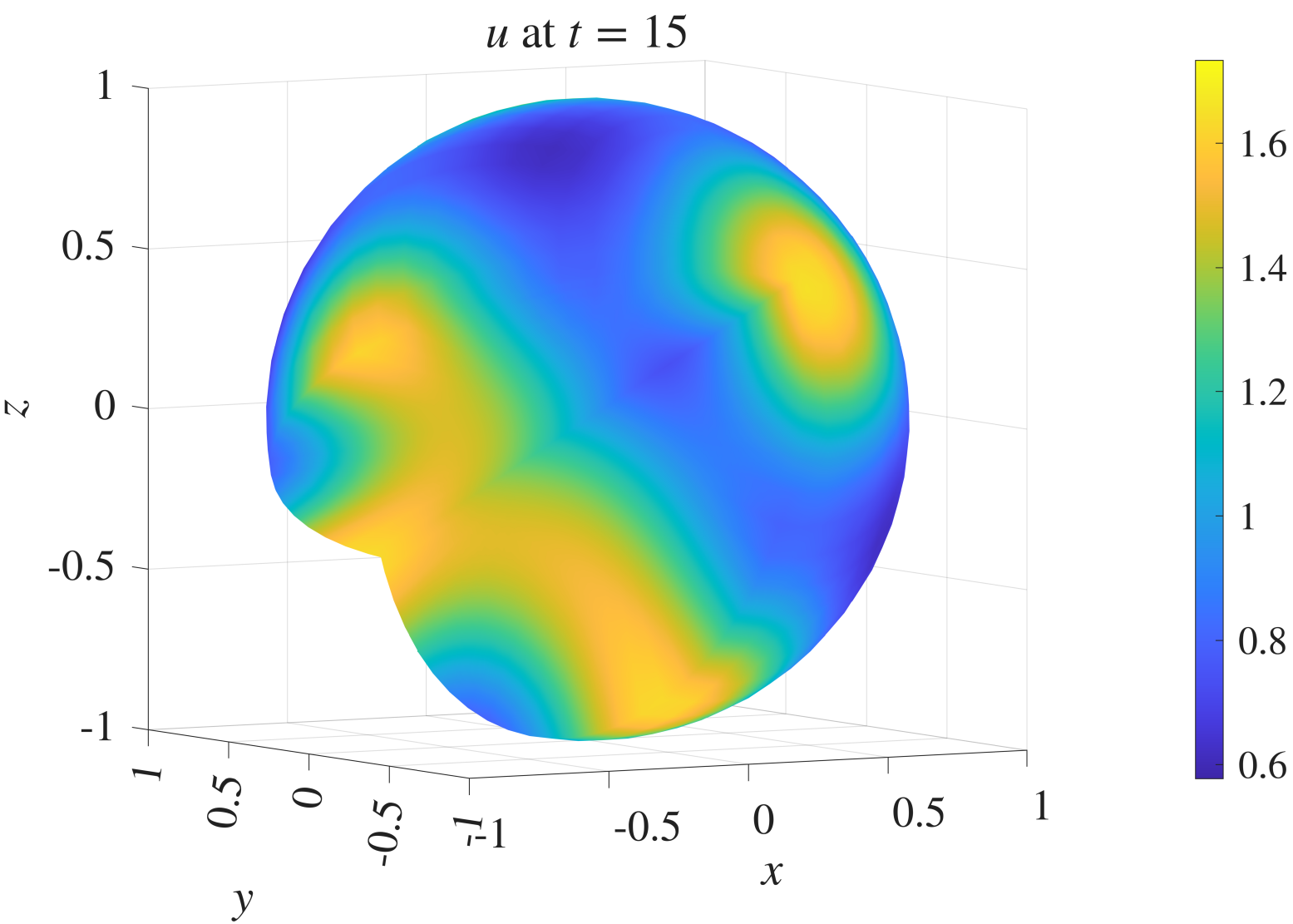}
  \includegraphics[scale=0.165]{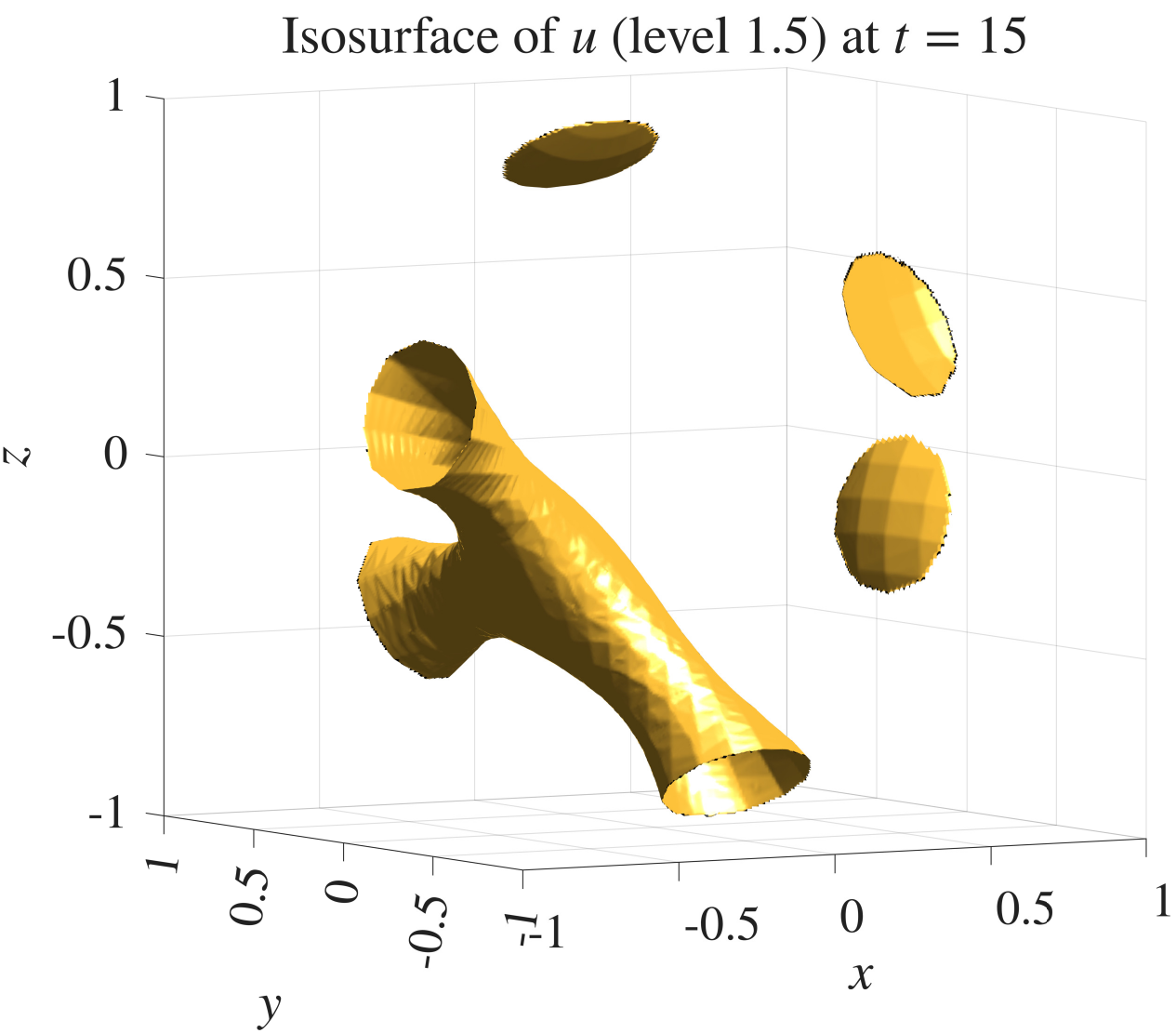}
  \includegraphics[scale=0.165]{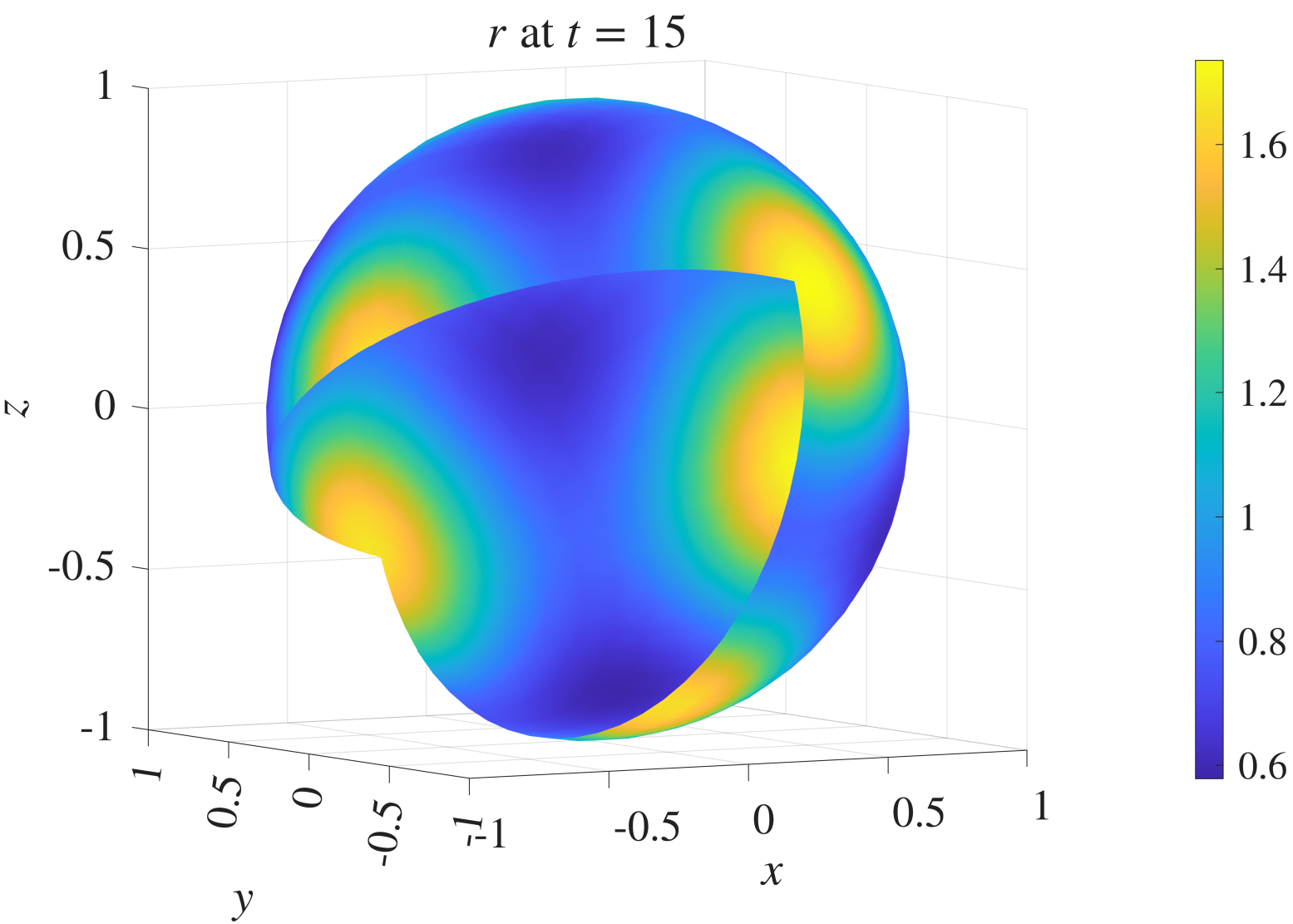}\\
  \includegraphics[scale=0.165]{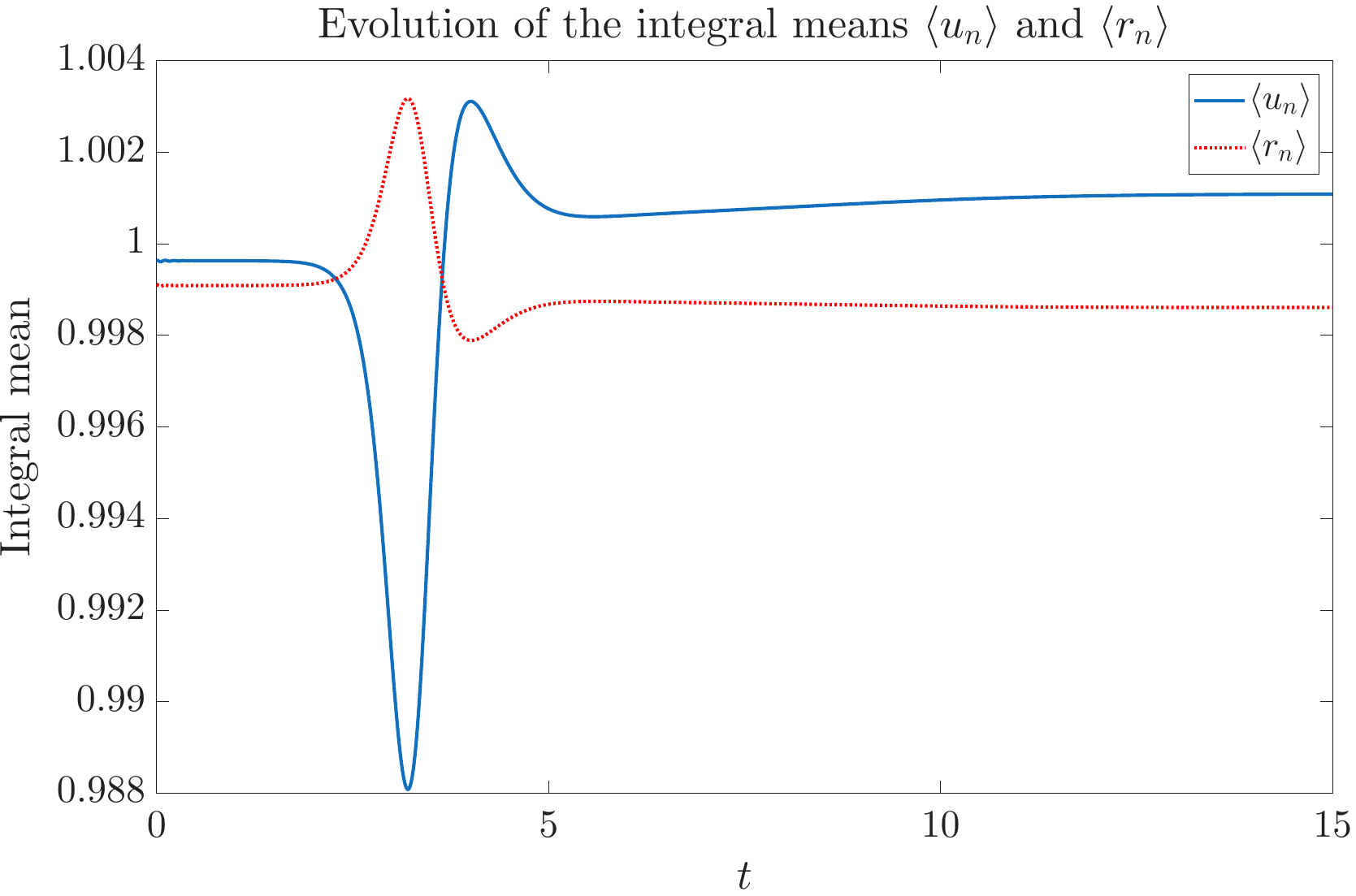}
  \includegraphics[scale=0.165]{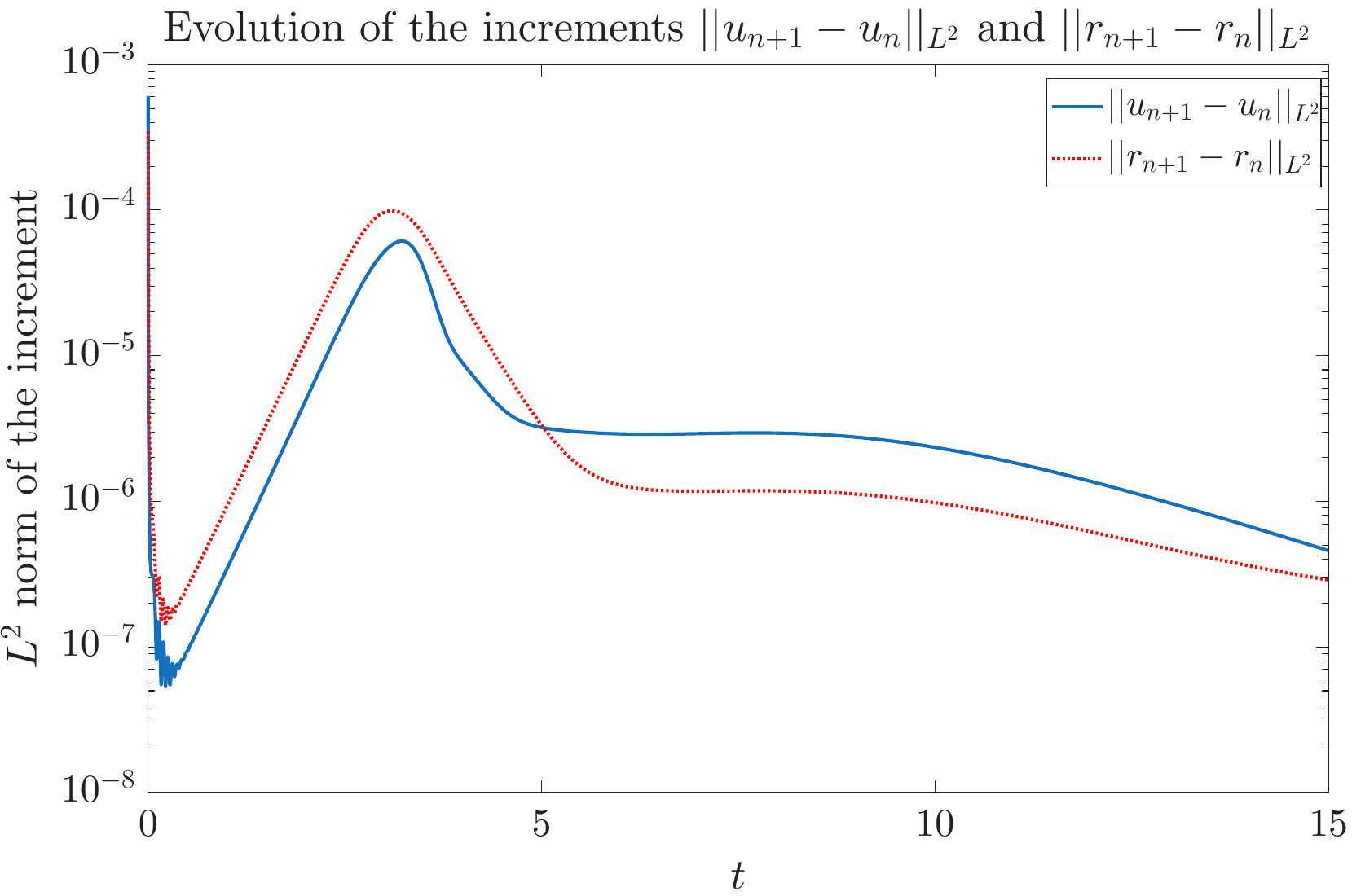}
  \caption{Results of the experiment in Section~\ref{sec:ball} using
    the adaptation of integrators~\eqref{eq:ee_sphere} and~\eqref{eq:ee_ball}
    to coupled systems.
    The DOF per
  direction are $(n_\rho,n_\theta,n_\phi)=(30,50,30)$
  (total number of DOF $2\cdot 45000 + 2\cdot 1500$),
  and the number of time steps is {\color{black}$m=150000$.
  Top: $u$ component (left),
  isosurface at level $1.5$ (center), and $r$ component (right)
  at the final time.
  Bottom: evolution of the integral means (left) and of the increments
  (right).
  The total simulation wall-clock time is about 4.8 minutes.}}
  \label{fig:plot_turing_ball_schnakenberg}
\end{figure}

The results of the simulation setting the DOF per direction
to $(n_\rho,n_\theta,n_\phi)=(30,50,30)$
and the number of time steps to {\color{black}$m=150000$}
are depicted in Figure~\ref{fig:plot_turing_ball_schnakenberg}.
Also in this more challenging three-dimensional case, we are able to
efficiently
obtain the expected pattern using our tensor-based proposed technique
(cf.~\cite[Fig.~3]{FMS23}). In fact,
the split exponential Euler method is able to retrieve the spatially
inhomogeneous
steady state (tunnels connecting three out of six spots of $u$, all
in correspondence
with spots in the $r$ component) in approximately {\color{black}4.8 minutes}.
As expected,
both the integral means of $u$ and $r$ stabilize to constant values
as time proceeds, {\color{black} while the increments decrease.}

\subsection{Bulk-surface DIB model on the cylinder}\label{sec:cylinder}
We conclude our set of experiments with another three-dimensional bulk-surface
model. In particular, we consider the BS-DIB model presented in~\cite{FSB24}
but, instead of choosing
a cube as the spatial domain $\Omega$, we take a cylinder. To this aim, we define
the boundary surface as $\partial\Omega = \Omega_\mathrm{B} \cup \Omega_\mathrm{L} \cup \Omega_\mathrm{T}$
where $\Omega_\mathrm{B}$ is the bottom disk of the cylinder, $\Omega_\mathrm{L}$
is its lateral surface, and $\Omega_\mathrm{T}$ it its top
disk (see Figure~\ref{fig:cylinder} for a graphical representation).
\begin{figure}[!ht]
  \centering
\includegraphics[scale=0.5]{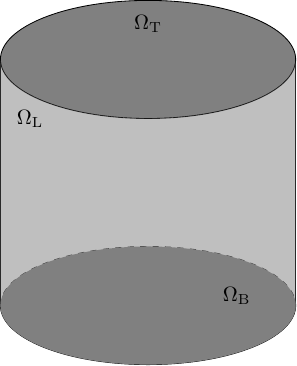}
\caption{Graphical representation of the spatial domain $\Omega$ for
  the experiment in Section~\ref{sec:cylinder}.}
\label{fig:cylinder}
\end{figure}

The governing equations are
\begin{subequations}\label{eq:bulkcyl}
\begin{equation}\label{eq:bulkcylsys}
  \left\{
    \begin{aligned}
    \partial_t u &= \Delta_\mathrm{C} u + b(u) && \text{on } \Omega, \\
    \partial_t v &= \delta \Delta_\mathrm{C} v + c(v) && \text{on } \Omega, \\
    \partial_t r &= \Delta_\mathrm{D} r +
    \zeta_1p(u,r,s) && \text{on } \Omega_\mathrm{B},\\
    \partial_t s &= \varepsilon \Delta_\mathrm{D} s +
    \zeta_1q(v,r,s) && \text{on } \Omega_\mathrm{B},\\
    u &= \alpha_2, \quad v = \beta_2 && \text{on } \Omega_\mathrm{T},\\
    \nabla u\cdot\mathbf{n} &= -\zeta_1\alpha_3p(u,r,s) && \text{on } \Omega_\mathrm{B},\\
    \nabla v\cdot\mathbf{n} &= -\zeta_1\beta_3q(v,r,s) && \text{on } \Omega_\mathrm{B},\\
    \nabla u\cdot\mathbf{n} &= 0, \quad \nabla v\cdot\mathbf{n} = 0 && \text{on } \Omega_\mathrm{L},\\
    \nabla r\cdot \mathbf{n}&=0,\quad \nabla s\cdot \mathbf{n} =0 &&
    \text{on } \partial\Omega_\mathrm{B},
    \end{aligned}
  \right.
\end{equation}
where
\begin{equation}\label{eq:bulkcylnl}
  \begin{aligned}
    b(u)&= -\alpha_1(u-\alpha_2),\\
    c(v)&= -\beta_1(v-\beta_2),\\
    p(u,r,s)&=\zeta_2u(1-s)r-\zeta_3r^3-\zeta_4(s-\zeta_5),\\
    q(v,r,s)&=\eta_1v(1+\eta_2r)(1-s)(1-\eta_3(1-s))-\eta_4(1+\eta_5r)s(1+\eta_3s).
  \end{aligned}
\end{equation}
\end{subequations}
The symbols $\Delta_\mathrm{C}$ and $\Delta_\mathrm{D}$ denote the Laplace
operator on the cylinder \eqref{eq:lapcyl} and the disk \eqref{eq:disk},
respectively.
The radius of the cylinder is set to {\color{black}$\rho_*=15$} and the height is set to
$z_*=25$. The parameters are set to
$\alpha_1=\alpha_2=\delta=\beta_1=\beta_2=\zeta_1=\zeta_3=1$,
{\color{black}$\alpha_3=\beta_3=0.1$},
$\varepsilon=20$, $\zeta_2 = 10$, {\color{black}$\zeta_4 = 30$},
$\zeta_5 = 0.5$, $\eta_1 = 3$,
$\eta_2 = 2.5$, $\eta_3 = 0.2$, $\eta_5 = 1.5$. Finally, the parameter
$\eta_4$ is set to $\beta_2\eta_1(1-\zeta_5)(1-\eta_3+\eta_3\zeta_5)/(\zeta_5(1+\eta_3\zeta_5))$
so that the spatially homogeneous equilibrium
$(u_\mathrm{e},v_\mathrm{e},r_\mathrm{e},s_\mathrm{e})
=(\alpha_2,\beta_2,0,\zeta_5)$
is subject to Turing instability. In fact, we take as initial condition
\begin{equation*}
  u_0 = u_\mathrm{e}, \quad
  v_0 = v_\mathrm{e}, \quad
  r_0 = r_\mathrm{e}+10^{-2}\cdot\mathcal{U}(0,1), \quad
  s_0 = s_\mathrm{e}+10^{-2}\cdot\mathcal{U}(0,1),
\end{equation*}
and integrate the system up to the final time {\color{black}$t_*=250$}
employing the adaptation of schemes~\eqref{eq:ee_disk}
and~\eqref{eq:ee_cylinder} to coupled systems.

The results of the simulation setting the DOF per direction
to {\color{black}$(n_\rho,n_\theta,n_z)=(140,140,25)$}
and the number of time steps to {\color{black}$m=10000$}
are displayed in Figures~\ref{fig:plot_turing_cylinder_DIB_R}
($r$ component {\color{black}and its stationary indicators})
and~\ref{fig:plot_turing_cylinder_DIB_U} ($u$ component).
The obtained pattern {\color{black}results in spots  both in the
  surface component $r$ and in $u$ at $z=0$, as expected}.
Then, analogously to the case
considered in~\cite[{\color{black}Experiment D4}]{FSB24}, the pattern
dissipates as
$z$ increases, eventually
keeping the constant value $\alpha_2=1$.
The overall wall-clock time for this simulation is approximately
{\color{black}4.6 minutes}.
\begin{figure}[!htb]
  \centering
  \includegraphics[scale=0.165]{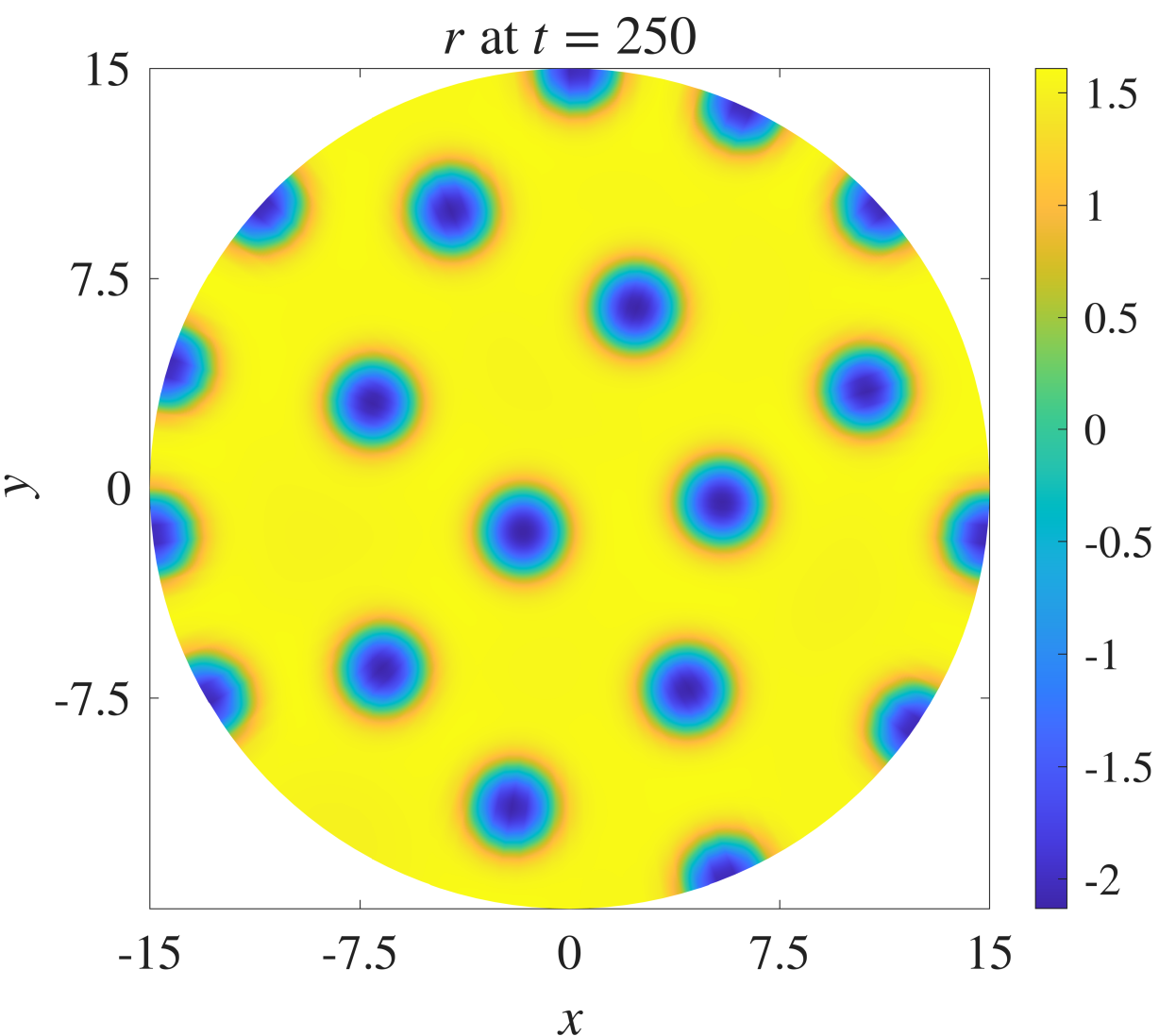}
  \includegraphics[scale=0.165]{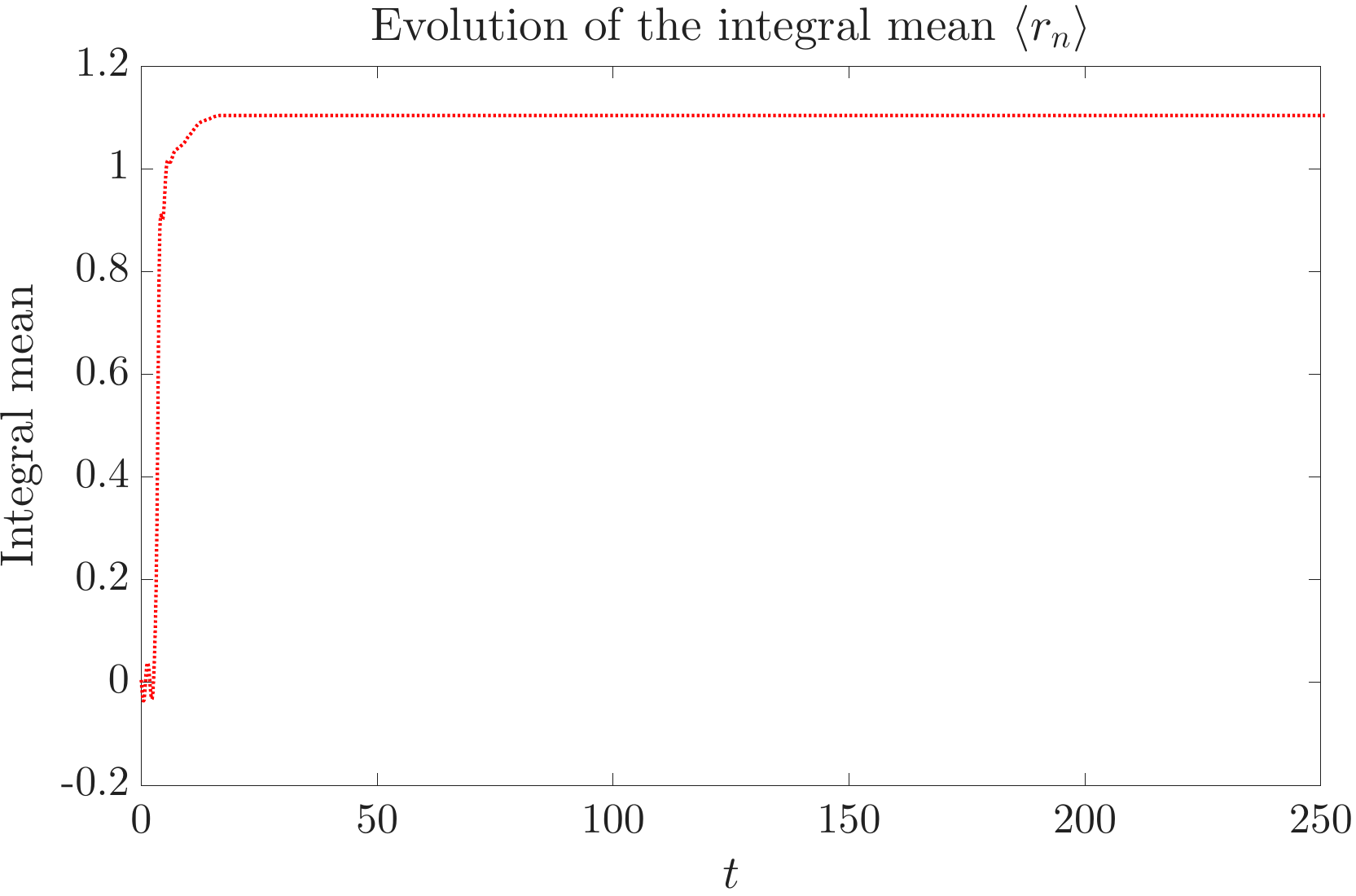}
    \includegraphics[scale=0.165]{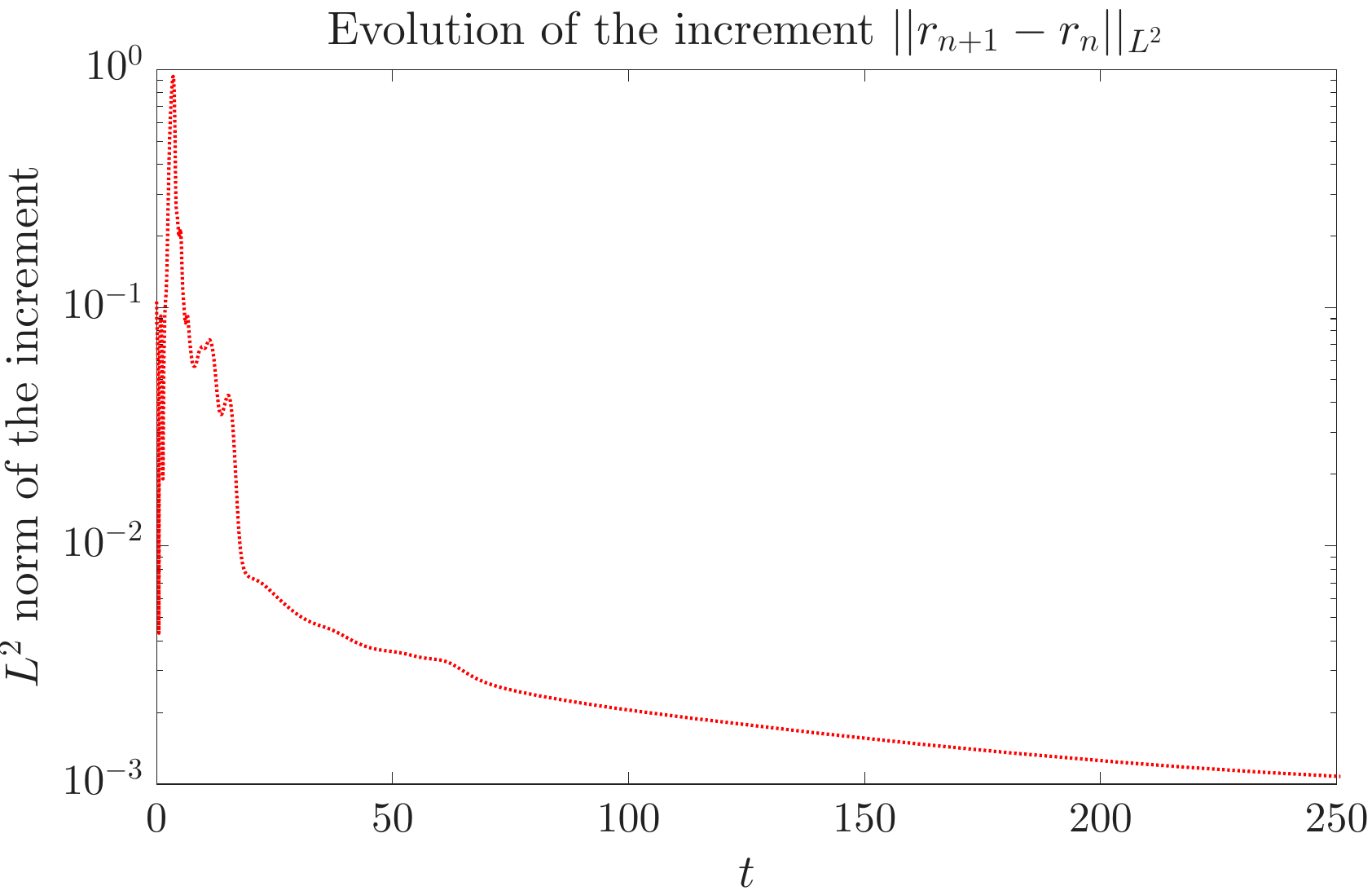}
    \caption{Results for the $r$ component of the experiment in
      Section~\ref{sec:cylinder} using
    the adaptation of integrators~\eqref{eq:ee_disk} and~\eqref{eq:ee_cylinder}
    to coupled systems. The DOF per
  direction are {\color{black}$(n_\rho,n_\theta,n_z)=(140,140,25)$}
  (total number of DOF $2\cdot 490000 + 2\cdot 19600$),
  and the number of time steps is {\color{black}$m=10000$. Left: $r$ component
  at the final time. Center:
  evolution of the integral mean. Right: evolution of the increment.
  The total simulation wall-clock time is about 4.6 minutes.}}
  \label{fig:plot_turing_cylinder_DIB_R}
\end{figure}
\begin{figure}[!htb]
  \centering
  \includegraphics[scale=0.165]{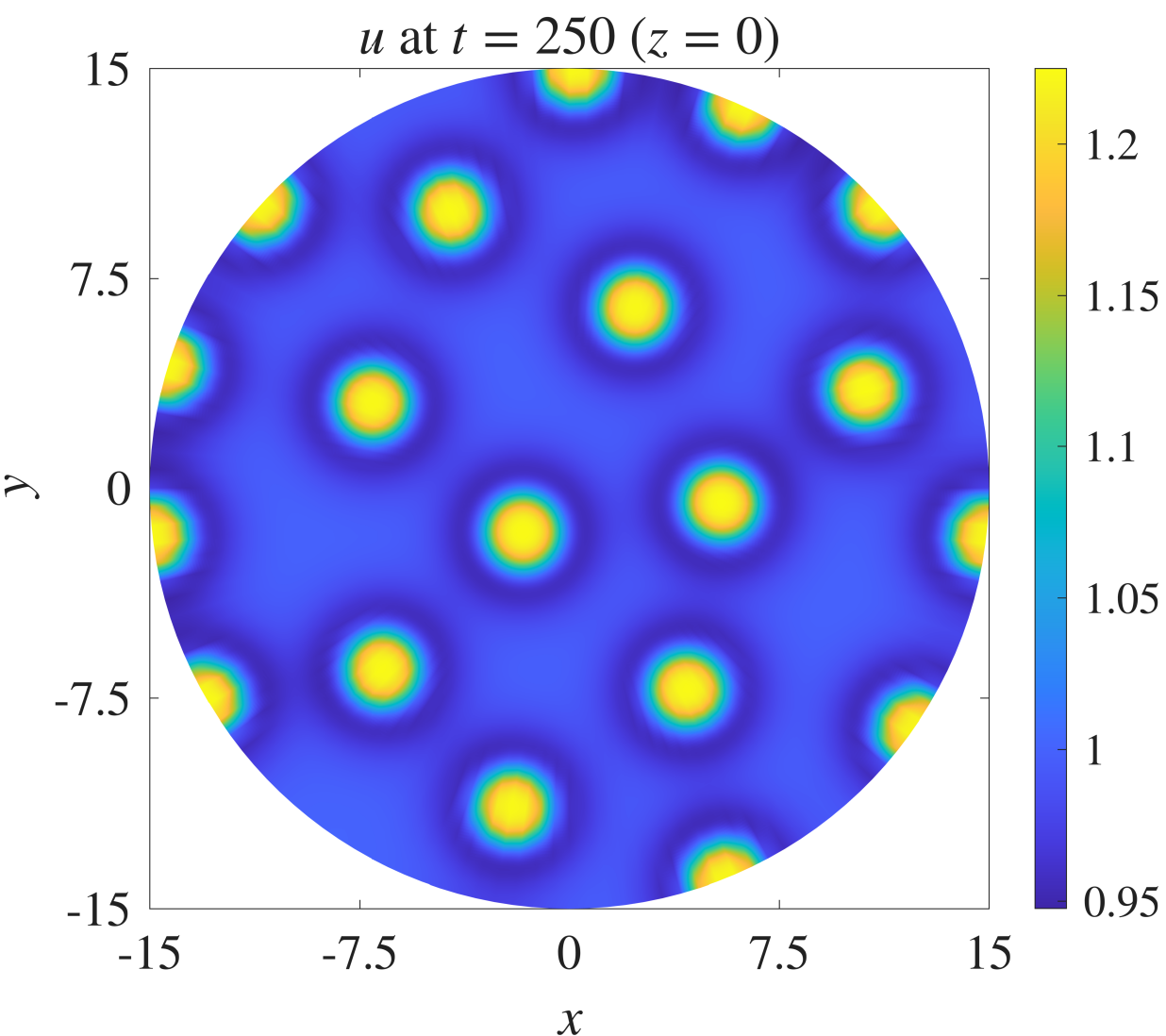}\hspace{1cm}
  \includegraphics[scale=0.165]{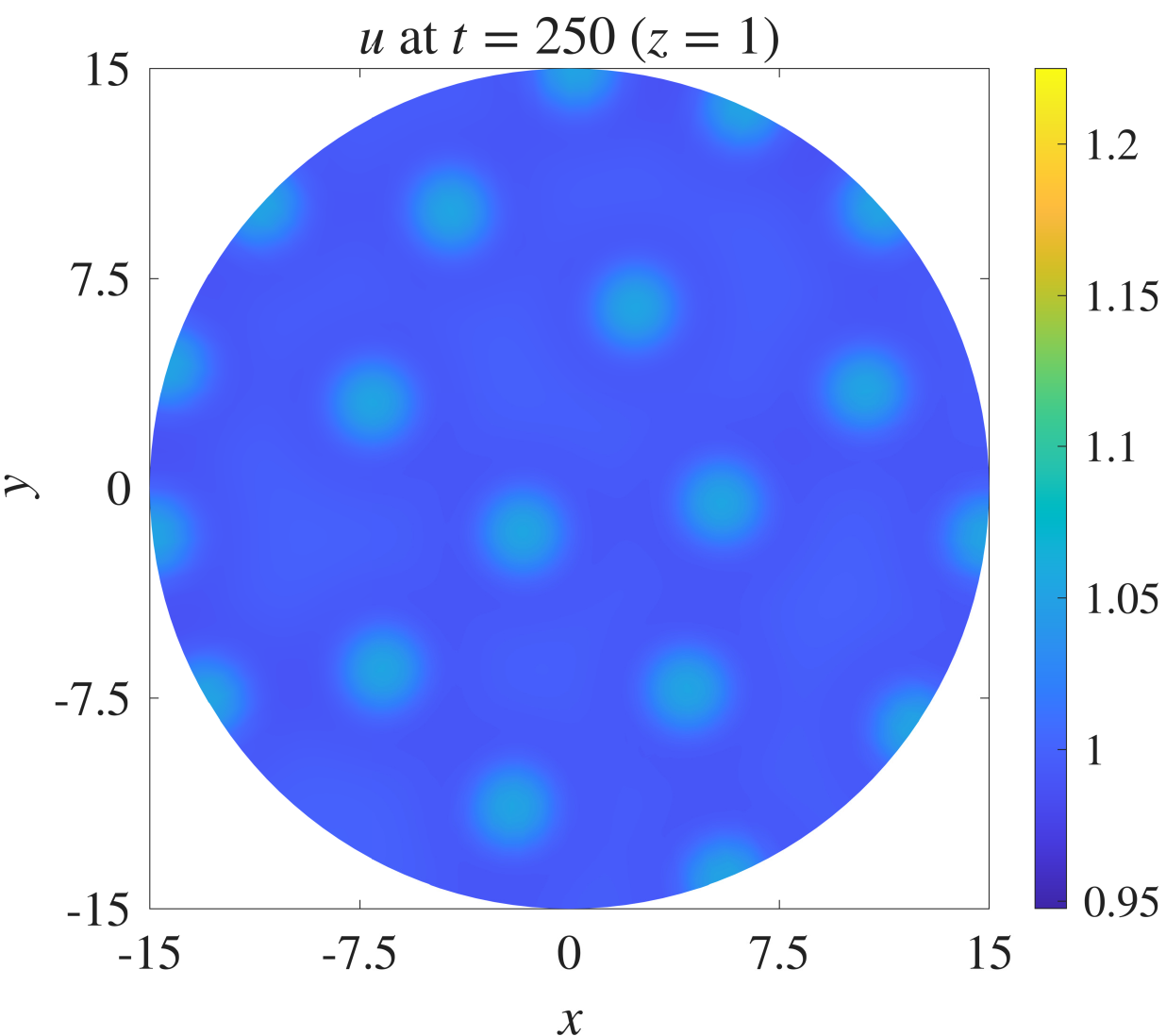}\\[2ex]
  \includegraphics[scale=0.165]{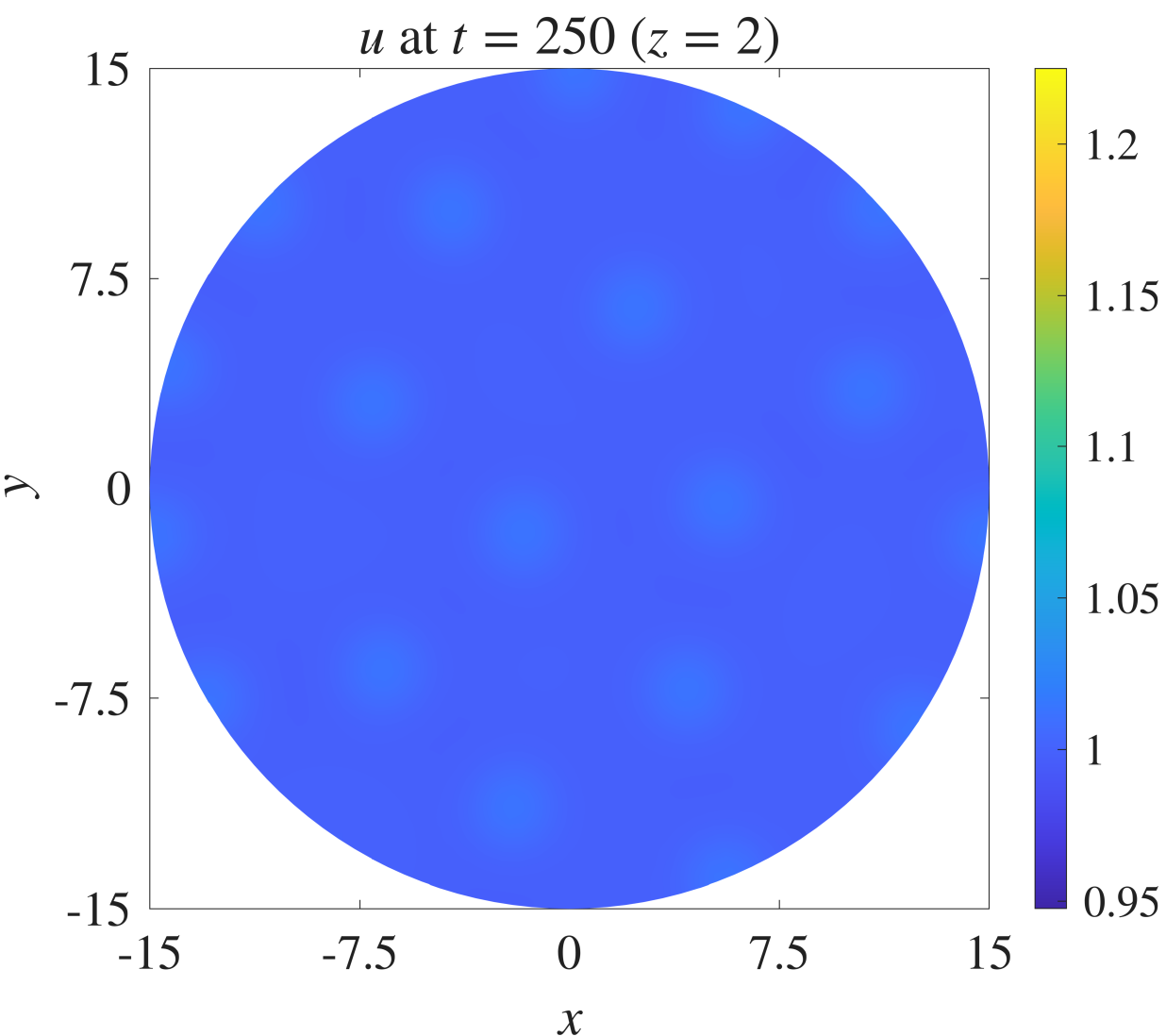}\hspace{1cm}
  \includegraphics[scale=0.165]{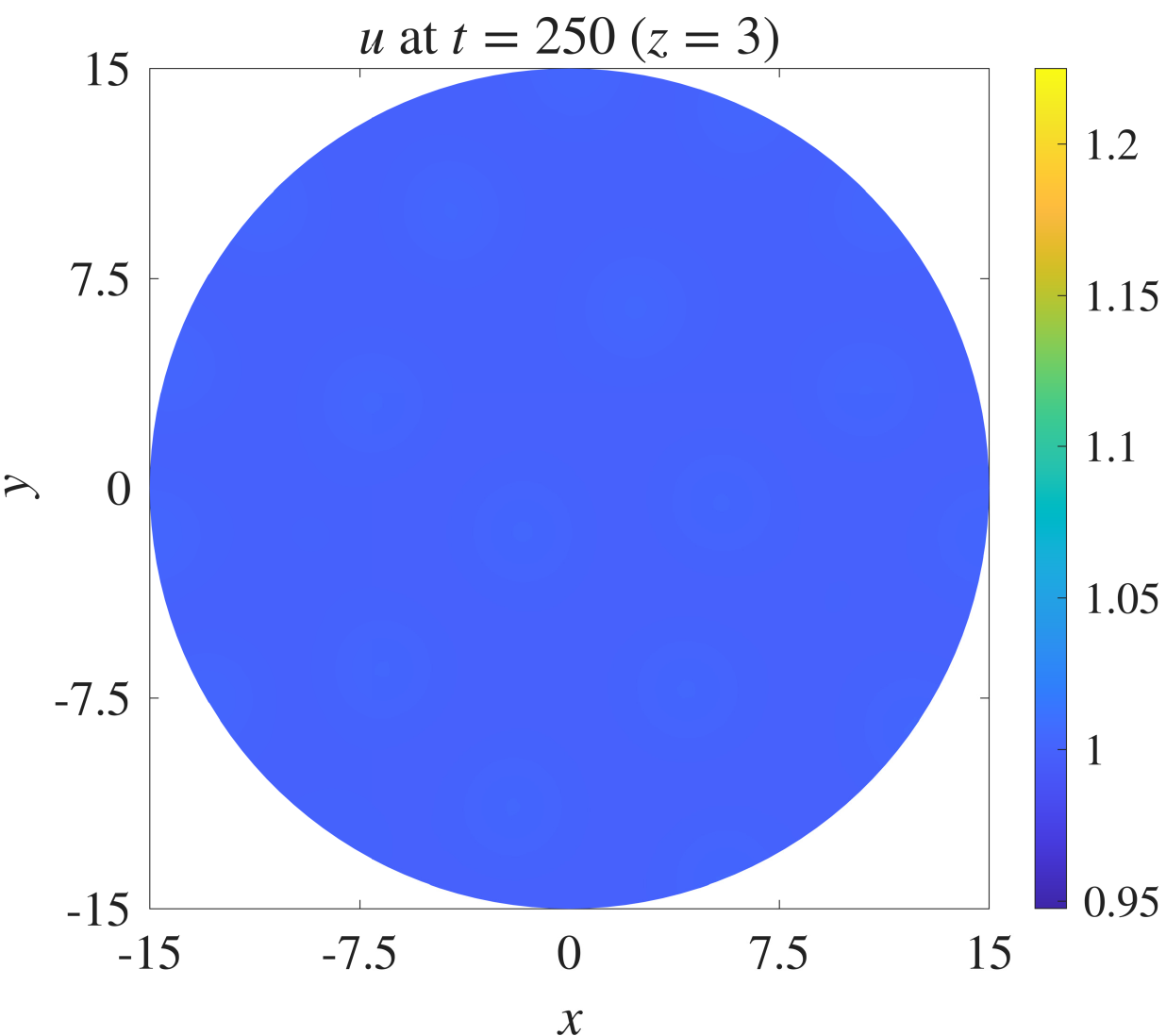}
  \caption{Results for the $u$ component of the experiment in
  Section~\ref{sec:cylinder} using the adaptation of
    integrators~\eqref{eq:ee_disk} and~\eqref{eq:ee_cylinder} to coupled
    systems, see also
    Figure~\ref{fig:plot_turing_cylinder_DIB_R}.
  }
  \label{fig:plot_turing_cylinder_DIB_U}
\end{figure}

\section{Conclusions}\label{sec:conclusions}
In this manuscript, we proposed efficient matrix and tensor realizations
of a novel split exponential Euler integrator for the solution of a class of
evolutionary
equations in common curvilinear coordinates (such as polar, spherical, and
cylindrical systems).
The chosen discretization guarantees stability and
allows the evaluation of the relevant
quantities by using high-performance level 3 BLAS operations, thanks to
the $\mu$-mode framework. Numerical experiments on representative
two- and three-dimensional diffusion--reaction systems for Turing pattern
formation demonstrate both the robustness and the efficiency of
the proposed approach, with computational times ranging
from {\color{black}half a second to less than 5 minutes} for three-dimensional
simulations involving more than $10^6$ degrees of freedom.
Promising directions for future research include {\color{black}the
  extension to higher-order time integrators,} a rigorous stiff
error analysis of the scheme, as well as its extension to more
general differential operators and alternative geometric settings.

\section{Acknowledgments}
The authors are members of the Gruppo Nazionale Calcolo
Scientifico-I\-sti\-tu\-to Na\-zio\-na\-le di Alta Matematica (GNCS-INdAM).
Fabio Cassini was partially funded by INdAM.

\appendix
\section{On the discretization matrix $A_\phi$}\label{sec:Aphi}
In this appendix, we prove the existence of a unique solution
to system~\eqref{eq:syscn}, which is related to the discretization matrix
$A_\phi$ in formula~\eqref{eq:A_phi}. Moreover, we show that the transformation
matrix of $A_\phi$
is weakly well conditioned in the 2-norm for sufficiently large $n$.

By substituting $h_\phi$ in the second equation of~\eqref{eq:syscn},
the unknown $\sigma$ can be found by computing the root of
\begin{equation*}
  f(x)=\cot\left(\frac{x\pi}{n-1+2x}\right)-\frac{2(n-1+2x)}{\pi}.
\end{equation*}
This nonlinear equation has a unique solution, since
$f(x)$ changes
sign in the interval $(0,1/2)$ and $f'(x)<0$. In fact,
$f(x)\rightarrow +\infty$ as $x\rightarrow 0$.
Then, by noticing that the argument
of the cotangent is always less than $\pi/2$, and by exploiting that $\cot(y)<1/y$
for $0<y<\pi/2$, we get $f(1/2)<0$. Finally, by direct computation we get
\begin{equation*}
  f'(x) = -\left(\frac{4}{\pi}+\frac{(n-1)\pi\mathrm{csc}^2(\pi x/(2x+n-1))}{(2x+n-1)^2}\right)<0.
\end{equation*}
Hence, $\sigma\in(0,1/2)$.

Concerning the conditioning, we first show that the
solution $\sigma$ of~\eqref{eq:syscn} tends to $1/2$
for large $n$, which means that the discretization resembles a half-point
discretization at both endpoints of the interval.
We already proved that $\sigma<1/2$. Then, using system~\eqref{eq:syscn} and
exploiting the inequality $\cot(y)>1/y+1/(y-\pi)$ for $0<y<\pi/2$, we get
\begin{equation*}
  \frac{2(n-1+2\sigma)}{\pi} > \frac{(n-1+2\sigma)}{\sigma\pi}
  + \frac{1}{\frac{\sigma\pi}{(n-1+2\sigma)}-\pi}
  \iff \sigma^2+(n-1)\sigma - (n-1)/2 > 0.
\end{equation*}
The relevant solution is
\begin{equation*}
  \sigma > \frac{-(n-1)+\sqrt{n^2-1}}{2}
\end{equation*}
and the right hand side tends to $1/2$ for $n$ large. Therefore, we conclude
by the sandwich theorem.
Thanks to this result, and using again the above inequalities for the cotangent
function, we can easily obtain that the condition number of the diagonal
symmetrization matrix of $A_\phi$ grows as $\sqrt{n}$.

\section{On the discretization matrix of
  $\Delta_\lambda$}\label{sec:A_lambda}
In this appendix, we show that the transformation
matrix of the discretization matrix of $\Delta_\lambda$ for $\lambda<0$
(see formula~\eqref{eq:A_lambda})
is weakly well conditioned in the 2-norm.
The element $\xi_\ell$ of the diagonal symmetrization matrix $\Xi$ is
\begin{equation*}
  \sqrt{\frac{(1-\lambda)^{1+\lambda}}{((2\ell-1)-\lambda)^{1+\lambda}}
  \frac{(\ell-1)!}{\prod_{i=1}^{\ell-1}(i-\lambda)}},\quad \ell\ge 2.
\end{equation*}
The smallest value is attained at $\ell=n$ and therefore the
largest value in $\Xi^{-1}$ is
\begin{equation*}
  \sqrt{\left(\frac{(2n-1)-\lambda}{1-\lambda}\right)^{1+\lambda}P_{n-1}(\lambda)},
  \quad   P_{n-1}(\lambda)=\frac{\prod_{i=1}^{n-1} (i-\lambda)}{(n-1)!}.
\end{equation*}
Taking the logarithm, we have
\begin{equation*}
  \begin{aligned}
  \log P_{n-1}(\lambda) &= \log (1-\lambda)+
  \log(1-\frac{\lambda}{2})\dots+\log(1-\frac{\lambda}{n-1})\\
  &<-\lambda(1+\frac{1}{2}+\dots+\frac{1}{n-1})<-\lambda(\ln (n-1)+1).
\end{aligned}
  \end{equation*}
Therefore,
\begin{equation*}
P_n(\lambda)< \rme^{-\lambda(\ln (n-1)+1)}=\mathcal{O}(n^{-\lambda}).
\end{equation*}
Hence, the largest element of $\Xi^{-1}$ grows as $\sqrt{n}$.
Since $\lVert \Xi \rVert_2 = 1$ we get the result.

\bibliography{CC26.bib}

@article{FS24,
  author = {M. Frittelli and I. Sgura},
  title = {Matrix-oriented {FEM} formulation for reaction-diffusion {PDE}s
on a large class of {2D} domains},
  journal = {Appl. Numer. Math.},
  year = {2024},
  volume = {200},
  pages = {286--308},
  doi = {https://doi.org/10.1016/j.apnum.2023.07.010}
}

@article{HS21,
  author = {Y. Hao and V. Simoncini},
  title = {Matrix equations solving of {PDE}s in polygonal domains using conformal
mappings},
  journal = {J. Numer. Math.},
  volume = {29},
  number = {3},
  pages = {221--244},
  year = {2021},
  doi = {httsp://doi.org/10.1515/jnma-2020-0035}
}

@article{MJKL16,
  author = {A. Mantzaflaris and B. J\"uttler and Khoromskij, B. N. and U. Langer},
  title = {Low rank tensor methods in {G}alerkin-based isogeometric analysis},
  journal = {Comput. Methods Appl. Mech. Engrg.},
  volume = {316},
  pages = {1062--1085},
  year = {2016},
  doi = {http://doi.org/10.1016/j.cma.2016.11.013}
}

@article{KB09,
  author = {Kolda, T. G. and Bader, B. W.},
  title = {Tensor decompositions and applications},
  journal = {SIAM Rev.},
  volume = {51},
  number = {3},
  pages = {455--500},
  year = {2009},
  doi = {https://doi.org/10.1137/07070111X}
}

@article{AMH11,
  author = {Al-Mohy, A. H. and Higham, N. J.},
  title = {Computing the {A}ction of the {M}atrix {E}xponential, with an {A}pplication
  to {E}xponential {I}ntegrators},
  journal = {SIAM J. Sci. Comput.},
  volume = {33},
  number = {2},
  pages = {488--511},
  year = {2011},
  doi = {https://doi.org/10.1137/100788860}
}

@article{LPR19,
  author = {Luan, V. T. and Pudykiewicz, J. A. and Reynolds, D. R.},
  title = {Further development of efficient and accurate time integration schemes for meteorological models},
  journal = {J. Comput. Phys.},
  volume = {376},
  pages = {817--837},
  year = {2019},
  doi = {https://doi.org/10.1016/j.jcp.2018.10.018}
}

@article{CCZ23,
  author = {M. Caliari and F. Cassini and F. Zivcovich},
  title = {{BAMPHI}: {M}atrix-free and transpose-free action of linear
            combinations of $\varphi$-functions from exponential integrators},
  journal = {J. Comput. Appl. Math.},
  volume = {423},
  pages = {114973},
  year = {2023},
  doi = {https://doi.org/10.1016/j.cam.2022.114973}
}

@article{HHHNLHC17,
  author = {Hern\'andez, D. and Herrera-Hern\'andez, E. C. and
  N\'u\~nez-L\'opez, M. and Hern\'andez-Coronado, H.},
  title = {Self-similar {T}uring patterns: {A}n anomalous diffusion consequence},
  journal = {Phys. Rev. E},
  volume = {95},
  pages = {022210},
  year = {2017},
  doi = {https://doi.org/10.1103/PhysRevE.95.022210}
}

@article{BKW18,
  author = {Bhatt, H.~P. and Khaliq, A.~Q.~M. and Wade, B.~A.},
  title = {Efficient {K}rylov-based exponential time differencing method
  in application to 3{D} advection-diffusion-reaction systems},
  journal = {Appl. Math. Comput.},
  volume = {338},
  pages = {260--273},
  year = {2018},
  doi = {https://doi.org/10.1016/j.amc.2018.06.025}
}

@article{S79,
  author = {Schnakenberg, J.},
  title = {Simple chemical reaction systems with limit cycle behaviour},
  journal = {J. Theor. Biol.},
  volume = {81},
  pages = {389--400},
  year = {1979},
  doi = {https://doi.org/10.1016/0022-5193(79)90042-0}
}

@article{GLRS19,
  author = {Gambino, G. and Lombardo, M.~C. and Rubino, G. and Sammartino, M.},
  title = {Pattern selection in the 2{D} {F}itz{H}ugh--{N}agumo model},
  journal = {Ric. Mat.},
  volume = {68},
  pages = {535--549},
  year = {2019},
  doi = {https://doi.org/10.1007/s11587-018-0424-6}
}

@article{CC25,
  author = {Caliari, M. and Cassini, F.},
  title = {Matrix- and tensor-oriented numerical schemes for the evolutionary
           space-fractional complex {G}inzburg--{L}andau equation},
  journal = {arXiv preprint arXiv:2510.21394},
  year = {2025},
  doi = {https://doi.org/10.48550/arXiv.2510.21394}
}

@article{CCEOZ22,
  author = {M. Caliari and F. Cassini and L. Einkemmer and
  A. Ostermann and F. Zivcovich},
  journal = {J. Comput. Phys.},
  pages = {110989},
  title = {A $\mu$-mode integrator for solving evolution
  equations in {K}ronecker form},
  volume = {455},
  year = {2022},
  doi = {https://doi.org/10.1016/j.jcp.2022.110989}
}

@article{CCZ23bis,
  author = {M. Caliari and F. Cassini and F. Zivcovich},
  title = {A $\mu$-mode {BLAS} approach for multidimensional tensor-structured
  problems},
  journal = {Numer. Algorithms},
  year = {2023},
  volume = {92},
  number = {4},
  pages = {2483--2508},
  doi = {https://doi.org/10.1007/s11075-022-01399-4}
}

@article{CC24bis,
  author = {Caliari, M. and Cassini, F.},
  title = {A second order directional split exponential integrator
           for systems of advection--diffusion--reaction equations},
  journal = {J. Comput. Phys.},
  year = {2024},
  volume = {498},
  pages = {112640},
  doi = {https://doi.org/10.1016/j.jcp.2023.112640}
}

@article{CC24,
  author = {Caliari, M. and Cassini, F.},
  title = {Direction splitting of $\varphi$-functions in exponential integrators
           for $d$-dimensional problems in {K}ronecker form},
  journal = {J. Approx. Softw.},
  year = {2024},
  volume = {1},
  number = {1},
  pages = {1},
  doi = {https://doi.org/10.13135/3103-1935/10813}
}

@article{VAB99,
  author = {Varea, C. and Arag\'on, J. L. and Barrio, R. A.},
  title = {Turing patterns on a sphere},
  journal = {Phys. Rev. E},
  volume = {60},
  number = {4},
  pages = {4588--4592},
  year = {1999},
  doi = {https://doi.org/10.1103/PhysRevE.60.4588}
}

@article{SMTT23,
  author = {Singh, S. and Mittal, R.~C. and Thottoli, S.~R. and Tamsir, M.},
  title = {High-fidelity simulations for {T}uring pattern formation in
           multi-dimensional {G}ray--{S}cott reaction-diffusion system},
  journal = {Appl. Math. Comput},
  volume = {452},
  pages = {128079},
  year = {2023},
  doi = {https://doi.org/10.1016/j.amc.2023.128079}
}

@article{ATGBM02,
  author = {Arag\'on, J. L. and Torres, M. and Gil, D. and Barrio, R.~A. and Maini, P.~K.},
  title = {Turing patterns with pentagonal symmetry},
  journal = {Phys. Rev. E},
  volume = {65},
  pages = {051913},
  year = {2002},
  doi = {https://doi.org/10.1103/PhysRevE.65.051913}
}

@article{FMS23,
  author = {Frittelli, M. and Madzvamuse, A. and Sgura, I.},
  title = {The bulk-surface virtual element method for reaction-diffusion
           {PDE}s: {A}nalysis and applications},
  journal = {Commun. Comput. Phys.},
  volume = {33},
  number = {3},
  pages = {733--763},
  year = {2023},
  doi = {https://doi.org/10.4208/cicp.OA-2022-0204}
}

@article{FSB24,
  author = {Frittelli, M. and Sgura, I. and Bozzini, B.},
  title = {Turing patterns in a 3{D} morpho-chemical bulk-surface reaction-diffusion system for battery modeling},
  journal = {Math. Eng.},
  volume = {6},
  number = {2},
  pages = {363--393},
  year = {2024},
  doi = {https://doi.org/10.3934/mine.2024015}
}

@article{LBFS17,
  author = {Lacitignola, D. and Bozzini, B. and Frittelli, M. and Sgura, I.},
  title = {Turing pattern formation on the sphere for a morphochemical reaction--diffusion model for electrodeposition},
  journal = {Commun. Nonlinear Sci. Numer. Simulat.},
  volume = {28},
  pages = {484--508},
  year = {2017},
  doi = {http://doi.org/10.1016/j.cnsns.2017.01.008}
}

@article{LHAMMHH24,
  author = {L\'opez, A. V. and Hern\'andez, D. and Aguilar-Madera, C. G. and
            Mart\'inez, R. C. and Herrera-Hern\'andez, E. C.},
  title = {Boundary conditions influence on {T}uring patterns under anomalous diffusion: {A} numerical exploration},
  journal = {Phys. D},
  volume = {470},
  pages = {134353},
  year = {2024},
  doi = {https://doi.org/10.1016/j.physd.2024.134353}
}

@article{L00,
  author = {M.-C. Lai},
  title = {A note on finite difference discretizations for {P}oisson equation on a disk},
  journal = {Numer. Methods Partial Differ. Equ.},
  volume = {17},
  number = {3},
  pages = {199--203},
  year = {2000},
  doi = {https://doi.org/10.1002/num.1}
}

@article{BT92,
  author = {L. Brugnano and D. Trigiante},
  title = {Tridiagonal matrices: {I}nvertibility and conditioning},
  journal = {Linear Algebra Appl.},
  volume = {166},
  pages = {131--150},
  year = {1992},
  doi = {https://doi.org/10.1016/0024-3795(92)90273-D}
}

@book{HV03,
  author = {W. Hundsdorfer and Verwer, J. G.},
  title = {Numerical Solution of Time-Dependent Advection-Diffusion-Reaction
  Equations},
  publisher = {Springer Berlin, Heidelberg},
  series = {Springer Series in Computational Mathematics},
  volume = {33},
  year = {2003},
  doi = {https://doi.org/10.1007/978-3-662-09017-6}
}

@book{M25,
  author = {Meurant, G.},
  title = {Hessenberg and Tridiagonal Matrices},
  year = {2025},
  publisher = {SIAM},
  address = {Philadelphia, PA},
  edition = {first},
  series = {Other Titles in Applied Mathematics},
  doi = {https://doi.org/10.1137/1.9781611978452}
}

@article{BW15,
  author = {B. Bialecki and L. Wright},
  title = {A fast direct solver for a fourth order
           finite difference scheme for {P}oisson's equation
           on the unit disc in polar coordinates},
  journal = {Numer. Algorithms},
  volume = {70},
  pages = {727--751},
  year = {2015},
  doi = {https://doi.org/10.1007/s11075-015-9971-z}
}

@book{B86,
  author = {Britton, N. F.},
  title = {Reaction-Diffusion Equations and Their Applications to Biology},
  year = {1986},
  publisher = {Academic Press},
  address = {New York, NY},
}

@book{M03,
  title={Mathematical Biology II: Spatial Models and Biomedical Applications},
  author={Murray, J. D.},
  volume={18},
  year={2003},
  publisher={Springer},
  edition = {third},
  series = {Interdisciplinary Applied Mathematics}
}

@book{S94,
  title={Shock waves and reaction--diffusion equations},
  author={Smoller, J.},
  edition = {second},
  year={1994},
  publisher={Springer},
  address = {New York, NY},
  series = {Grundlehren der mathematischen Wissenschaften},
  doi = {https://doi.org/10.1007/978-1-4612-0873-0}
}

@book{HW96,
  author = {E. Hairer and G. Wanner},
  title = {Solving Ordinary Differential Equations II: Stiff and
  Differential Algeraic Problems},
  edition = {second},
  year = {1996},
  publisher = {Springer Berlin},
  series = {Springer Series in Computational Mathematics},
  volume = {14},
  doi = {https://doi.org/10.1007/978-3-642-05221-7}
}

@article{ARW95,
  author = {Ascher, U. M. and Ruuth, S. J. and Wetton, B. T. R.},
  title = {Implicit-{E}xplicit {M}ethods for {T}ime-{D}ependent
           {P}artial {D}ifferential {E}quations},
  journal = {SIAM J. Numer. Anal.},
  volume = {32},
  number = {2},
  pages = {797--823},
  year = {1995},
  doi = {https://doi.org/10.1137/0732037}
}

@article{HO10,
  author = {M. Hochbruck and A. Ostermann},
  title = {Exponential integrators},
  journal = {Acta Numer.},
  volume = {19},
  pages = {209--286},
  year = {2010},
  doi = {https://doi.org/10.1017/S0962492910000048}
}

@article{DASS20,
  author = {M. C. D'Autilia and I. Sgura and V. Simoncini},
  title = {Matrix-oriented discretization methods for reaction--diffusion {PDE}s: {C}omparisons and applications},
  journal = {Comput. Math. Appl.},
  volume = {79},
  number = {7},
  pages = {2067--2085},
  year = {2020},
  doi = {https://doi.org/10.1016/j.camwa.2019.10.020}
}

@article{BS17,
  author = {M. Benzi and V. Simoncini},
  title = {Approximation of functions of large matrices with {K}ronecker
  structure},
  journal = {Numer. Math.},
  year = {2017},
  volume = {135},
  pages = {1--26},
  doi = {https://doi.org/10.1007/s00211-016-0799-9}
}

@article{CK20,
  author = {M. Chen and D. Kressner},
  title = {Recursive blocked algorithms for linear systems with {K}ronecker product structure},
  journal = {Numer. Algorithms},
  volume = {84},
  number = {3},
  pages = {1199--1216},
  year = {2020},
  doi = {https://doi.org/10.1007/s11075-019-00797-5}
}

@article{C24,
  author = {F. Cassini},
  title = {Efficient third-order tensor-oriented directional splitting for exponential integrators},
  journal = {Electron. Trans. Numer. Anal.},
  year = {2024},
  volume = {60},
  pages = {520--540},
  doi = {https://doi.org/10.1553/etna_vol60s520}
}

@article{KST99,
  author = {D. Kulkarni and D. Schmidt and S.-K. Tsui},
  title = {Eigenvalues of tridiagonal pseudo-{T}oeplitz matrices},
  journal = {Linear Algebra Appl.},
  volume = {297},
  pages = {63--80},
  year = {1999},
  doi = {https://doi.org/10.1016/S0024-3795(99)00114-7}
}

@article{VL00,
  author = {Van Loan, C. F.},
  title = {The ubiquitous {K}ronecker product},
  journal = {J. Comput. Appl. Math.},
  volume = {123},
  pages = {85--100},
  year = {2000},
  doi = {https://doi.org/10.1016/S0377-0427(00)00393-9}
}

@article{GRT18,
  author = {S. Gaudreault and G. Rainwater and M. Tokman},
  title = {{KIOPS: A} fast adaptive {K}rylov subspace solver
  for exponential integrators},
  journal = {J. Comput. Phys.},
  volume = {372},
  pages = {236--255},
  year = {2018},
  doi = {https://doi.org/10.1016/j.jcp.2018.06.026}
}
\bibliographystyle{plain}
\end{document}